\documentclass{amsart}
\usepackage{amssymb,amsmath}
\usepackage{hyperref}
\usepackage{appendix}
\usepackage{paralist}
\usepackage{graphics}
\usepackage{epsfig}
\usepackage{graphicx}
\usepackage{epstopdf}
\usepackage[]{algorithm2e}
\usepackage{graphicx,amsxtra,amssymb, tikz}
\usepackage{amsmath}
\usepackage{epsfig}
\usepackage{hyperref}
\usepackage{appendix}
\usepackage{setspace}
\usepackage{MnSymbol,wasysym}
\newtheorem{thm}{Theorem}[section]

\newtheorem{lem}{Lemma}[section]

\newtheorem{proposition}[thm]{Proposition}

\numberwithin{equation}{section}
\newcommand{\nset}{{\mathbb N}}

\newcommand{\ssy}{\scriptscriptstyle}

\newcommand{\half}{\frac{1}{2}}

\newcommand{\dspace}{{\mathfrak X}_h}

\newcommand{\bee}{\begin{equation}}
\newcommand{\eee}{\end{equation}}
\begin{document}
\title[]
{A linear implicit finite difference discretization\\
of the Schr{\"o}dinger-Hirota Equation}
\author[]{Georgios E. Zouraris$^{\ddag}$}
\thanks
{$^{\ddag}$ Department of Mathematics and Applied Mathematics,
University of Crete, PO Box 2208, GR-710 03 Heraklion, Crete, Greece.}
\subjclass{65M12, 65M60}
\keywords {Schr{\"o}dinger-Hirota equation, Hirota equation, linear implicit time stepping,
finite differences, periodic boundary conditions, optimal order error estimates, Bright Solitons}
%
%
%
%
\begin{abstract}
A linear implicit finite difference method is proposed
for the approximation of the solution to a periodic, initial value problem
for a Schr{\"o}dinger-Hirota equation.
Optimal, second order convergence in the discrete $H^1-$norm is proved,
assuming that $\tau$, $h$ and
$\tfrac{\tau^4}{h}$ are sufficiently small,
where $\tau$ is the time-step and $h$ is
the space mesh-size.
The efficiency of the proposed method is verified 
by results from numerical experiments.
\end{abstract}
\maketitle
%
%
%
%
%
%
%
%
\section{Introduction}
\subsection{Formulation of the problem}
For $T>0$ and ${\sf L}>0$, we consider the following periodic initial value problem:
find $\phi=\phi(t,x):[0,T]\times{\mathbb R}\rightarrow{\mathbb C}$ which is ${\sf L}-$periodic
on ${\mathbb R}$ and such that:
\begin{gather}
\phi_t={\rm i}\,\rho\,\phi_{xx}-\sigma\,\phi_{xxx}
-3\,\alpha\,|\phi|^2\,\phi_x +{\rm i}\,\delta\,|\phi|^2\,\phi+f
\quad\text{\rm on}\ \ (0,T]\times{\mathbb R},
\label{HirotaPeriodic}\\
\phi (0,x)=\phi_0(x)\quad\forall\,x\in{\mathbb R},
\label{HirotaPeriodic_b}
\end{gather}
where: $\rho$, $\sigma$, $\alpha$ and $\delta$ are real
constants, $\phi_0=\phi_0(x):{\mathbb R}\rightarrow{\mathbb C}$ is an
${\sf L}-$periodic function 
and $f=f(t,x):[0,T]\times{\mathbb R}\rightarrow{\mathbb C}$
is a function which is  ${\sf L}-$periodic on ${\mathbb R}$. 
\par
The nonlinear partial differential equation \eqref{HirotaPeriodic} is known as:
`{\it the cubic nonlinear Schr{\"o}dinger equation}' (cNLS) \cite{Chiao}
when $\alpha=\sigma=0$, `{\it the Hirota equation}' (H) \cite{Hirotar}
when
\begin{equation}\label{Hirota_Balance}
\rho\,\alpha=\sigma\,\delta
\end{equation}
and `{\it the complex modified Korteweg-de Vries equation}' (cmKdV) \cite{Taha}
when $\rho=\delta=0$. Since the (cmKdV) equation is a special case
of the  (H) equation, we adopt, for equation \eqref{HirotaPeriodic}, the name
Schr{\"o}dinger-Hirota (SH) equation (cf. \cite{Biswas}).
The (SH) equation is widely used in the description of the propagation of
optical solitons in a dispersive optical fiber
(see, e.g., \cite{Hasewaga}, \cite{Karan}), and in the modeling of the motion
of vortex filaments (see, e.g., \cite{Fuku}, \cite{Karan}).
%
%
%
%
%
\par
For existence and uniqueness results in the homogeneous case we refer the reader to \cite{Karan}.
There, it is shown that: i) if $\phi_0\in H^{\ell}_{\rm per}(0,{\sf L})$ for any integer $\ell\ge2$,
then there exists $T>0$ such that the problem has a unique solution $\phi\in C([0,T],H^{\ell}_{\rm per}(0,{\sf L}))$
with $\phi_t\in C([0,T],H_{\rm per}^{\ell-3}(0,{\sf L}))$ (see Theorem~2.1 in \cite{Karan}),
and ii) if $\phi_0\in H^3_{\rm per}(0,{\sf L})$, $\rho\not=0$ and \eqref{Hirota_Balance} is satisfied,
then 
$\phi\in C([0,+\infty),H^2_{\rm per}(0,{\sf L}))\cap C([0,+\infty),H^1_{\rm per}(0,{\sf L}))$
(see Theorem~2.5 in \cite{Karan}),
i.e., there is no finite time blow up for the solution and its first space derivate.  The latter results,
are based on the following conservation properties:
\begin{equation}\label{conserva_1}
\int_0^{\ssy{\sf L}}|\phi(t,x)|^2\,dx=\int_0^{\ssy{\sf L}}|\phi_0(x)|^2\,dx\quad\forall\,t\in[0,T],
\end{equation}
and, when \eqref{Hirota_Balance} holds, 
\begin{equation}\label{conserva_2}
\rho\,\int_0^{\ssy{\sf L}}|\phi_x(t,x)|^2\,dx-\tfrac{\delta}{2}\,\int_0^{\ssy{\sf L}}|\phi(t,x)|^4\,dx=
\rho\,\int_0^{\ssy{\sf L}}|\phi_0'(x)|^2\,dx-\tfrac{\delta}{2}\,\int_0^{\ssy{\sf L}}|\phi_0(x)|^4\,dx
\quad\forall\,t\in[0,T].
\end{equation}
However, it is easily seen that there exist unique, smooth,  special solutions to the homogeneous
problem \eqref{HirotaPeriodic}-\eqref{HirotaPeriodic_b} for any choice of
the parameters $\rho$, $\alpha$, $\delta$ and $\sigma$
(see Section~\ref{specialsolutions}).
\par
In the paper at hand, we focus on the numerical approximation of the solution to the problem
\eqref{HirotaPeriodic}-\eqref{HirotaPeriodic_b}. In particular, we propose a new
linear implicit finite difference method, the convergence of which is ensured by
providing an optimal, second order, error estimate. For the needs of the convergence analysis, we will
assume that the problem above admits a unique solution that is sufficiently smooth.
%
%
%
\subsection{The Finite Difference Method (FDM)}\label{TheFDMethod}
\par
Let ${\mathbb N}$ be the set of all positive integers, $x_a$, $x_b\in{\mathbb R}$ with $x_b-x_a={\sf L}$,
and $N$, $J\in{\mathbb N}$. Then, we define a uniform partition
of the time interval $[0,T]$ with time-step $\tau:=\tfrac{T}{N}$
and nodes $t_n:=n\,\tau$ for $n=0,\dots,N$, and
a uniform partition of ${\mathbb R}$ with mesh-width $h:=\tfrac{{\sf L}}{J}$
and nodes  $x_j:=x_a+jh$ for $j\in{\mathbb Z}$. Also, we introduce
the discrete space:
\begin{equation*}\label{discper}
\dspace:=\left\{\,(\psi_j)_{j\in\mathbb{Z}}\in{\mathbb C}^{\infty}:
\quad \psi_j=\psi_{j+J}\quad\forall\,j\in{\mathbb Z}\,\right\},
\end{equation*}
a discrete product operator
$\cdot\otimes\cdot:\dspace\times\dspace\rightarrow\dspace$
by
\begin{equation*}
(v{\otimes}w)_j=v_j\,w_j\quad\forall\,j\in{\mathbb Z},
\quad\forall\,v,w\in\dspace,
\end{equation*}
a discrete Laplacian operator $\Delta_h:\dspace\rightarrow\dspace$
by
\begin{equation*}
\Delta_hv_j:=\tfrac{v_{j-1}-2v_j+v_{j+1}}{h^2}
\quad\forall\,j\in{\mathbb Z},\quad\forall\,v\in\dspace,
\end{equation*}
a discrete space derivative operator $\partial_h:\dspace\rightarrow\dspace$
by
\begin{equation}\label{discrete_D1a}
\partial_hv_j:=\tfrac{v_{j+1}-v_{j-1}}{2h}
\quad\forall\,j\in{\mathbb Z},\quad\forall\,v\in\dspace
\end{equation}
and a discrete space average operator ${\mathcal A}_h:\dspace\rightarrow\dspace$ by
\begin{equation}\label{fyrom1}
{\mathcal A}_hv_j:=\tfrac{1}{2}\,(v_{j+1}+v_{j-1})
\quad\forall\,j\in{\mathbb Z},\ \ \forall\,v\in\dspace.
\end{equation}
In addition, we define the space
${\sf C}_{\mathrm{per}}:=\left\{\,v\in C({\mathbb R};{\mathbb C}):\ \
v(x+{\sf L})=v(x)\quad\forall\,x\in{\mathbb R}\,\right\}$ and
the o\-pe\-rator $\Lambda_h:{\sf C}_{\mathrm{per}}\rightarrow\dspace$
by $(\Lambda_hv)_j:=v(x_j)$ for $j\in{\mathbb Z}$ and
$v\in{\sf C}_{\mathrm{per}}$,
and we set $\phi^{n}:=\Lambda_h(\phi(t_{n},\cdot))$ for $n=0,\dots,N$.
For $\ell\in\nset$ and for any function $g:{\mathbb
C}^{\,\ell}\rightarrow{\mathbb C}$ and any
$w=(w^1,\dots,w^{\ell})\in(\dspace)^{\ell}$, we define
$g(w)\in\dspace$ by
$(g(w))_j:=g\left(w_j^1,\dots,w^{\ell}_j\right)$
for $j\in{\mathbb Z}$.
\par
For $n=0,\dots,N$, the proposed linear implicit finite difference method
constructs, recursively, an approximation $\Phi^n\in\dspace$ of $\phi^n$
following the steps below:
\par
{\tt Step 1}: Set
\begin{equation}\label{fidifin}
\Phi^0:=\phi^0.
\end{equation}
\par
{\tt Step 2}: Find $\Phi^1\in\dspace$ such that
\begin{equation}\label{scheme11}
\begin{split}
\tfrac{\Phi^1-\Phi^0}{\tau}=&\, {\rm
i}\,\rho\,\Delta_{h}\left(\tfrac{\Phi^1+\Phi^0}{2}\right)
-\sigma\,\partial_{h}\Delta_h\left(\tfrac{\Phi^1+\Phi^0}{2}\right)\\
&-3\,\alpha\,{\mathcal A}_h\left[\,|\Phi^0|^2
\otimes\partial_h\left(\tfrac{\Phi^1+\Phi^0}{2}\right)\,\right]
+{\rm i}\,\delta\,\left[|\Phi^0|^2\otimes
\left(\tfrac{\Phi^1+\Phi^0}{2}\right)\right]+F^{\half},
\end{split}
\end{equation}
where $F^{\half}\in\dspace$ with $(F^{\half})_j:=f(\frac{\tau}{2},x_j)$ for $j\in{\mathbb Z}$.
\par
{\tt Step 3}: For $n=1,\dots,N-1$, find $\Phi^{n+1}\in\dspace$
such that
\begin{equation}\label{scheme1}
\begin{split}
\tfrac{\Phi^{n+1}-\Phi^{n-1}}{2\,\tau}=&\,
{\rm i}\,\rho\,\Delta_h\left(\tfrac{\Phi^{n+1}+\Phi^{n-1}}{2}\right)
-\sigma\,\partial_h\Delta_h\left(\tfrac{\Phi^{n+1}+\Phi^{n-1}}{2}\right)\\
&-3\,\alpha\,{\mathcal A}_h\left[\,|\Phi^n|^2
\otimes\partial_h\left(\tfrac{\Phi^{n+1}+\Phi^{n-1}}{2}\right)\,\right] +{\rm
i}\,\delta\,\left[|\Phi^n|^2\otimes
\,\left(\tfrac{\Phi^{n+1}+\Phi^{n-1}}{2}\right)\right]+F^{n},
\end{split}
\end{equation}
where $F^{n}\in\dspace$ with $(F^{n})_j:=f(t_n,x_j)$ for $j\in{\mathbb Z}$. 
\par
Thus, at every time step, the computation of the finite difference approximations above,
requires the solution of a linear system of algebraic
equations the matrix of which is cyclic pentadiagonal.
\subsection{Overview and references}
%
%
%
%
%
%
%
%
%
%
\par

\par
The finite difference method formulated and computationally tested in \cite{Fei}
stands out among the known linear implicit methods for the
discretization of the (cNLS) equation, since it satisfies a discrete analogue of 
\eqref{conserva_1} and \eqref{conserva_2}.
Both discrete conservation laws ensure that
the finite difference approximations are uniformly
bounded in the discrete $L^{\infty}-$norm,
which leads to an optimal order error estimate 
(see, e.g., \cite{Geronti}, \cite{Georgios1}).
\par
The (FDM) we propose, for the numerical treatment of the
solution to the (SH) equation, is an extension of the method proposed in \cite{Fei}. However,  we
are not able to show that the (FDM) appro\-xi\-mations are uniformly bounded in the discrete
$W^{1,\infty}-$norm, which is necessary in handling the nonlinearities of the (SH) equation.
Therefore, the only choice left is to work with an auxiliary modified scheme.
\par
Using an idea from \cite{Georgios1}, first we define an operational mollifier 
depending on a positive para\-me\-ter $\lambda$ and the discrete $W^{1,\infty}-$norm
(see Section~\ref{section31}), and then we formulate a Modified Finite Difference Scheme
(MFDS) by mollifying properly the nonlinear terms of the (FDM) (see Section~\ref{Here_Is_MFD}).
Assuming that $\tau$ is small enough and $\lambda$ large enough, for the non-computable
(MFDS) approximations first we show that are well-defined (see Proposition~\ref{ModExist})
and then we establish an optimal, second order error estimate in a discrete $H^1-$norm
(see Theorem~\ref{Conv_Of_Mod}), which,
after applying a discrete Sobolev inequality, yields a convergence result in the discrete
$W^{1,\infty}-$norm.
Letting $h$ and $\tau^4\,h^{-1}$ to be small enough,  the latter convergence result yields that the discrete
$W^{1,\infty}-$norm distance of the (MFDS) approximations from the exact solution to the problem is lower
than $\lambda$. Finally, due to the special structure of the mollifier, we are able to conclude that  the (MFDS)
approximations are also (FDM) approximations and that the (FDM) approximations are unique
(see Theorem~\ref{ConvergenceRate}). Thus, the (FDM) inherits the convergence properties of the (MFDS).
%
%
%
%
\par
Numerical investigations of the solution to the  (SH) equation has been reported in
\cite{Karan} and \cite{Biswas} with no description of the numerical method used.
%
%
%
Also, numerical methods for the approximation of the solution to the (cmKdV) equation
has been proposed in \cite{Taha}, \cite{AlHarbi} and \cite{Raslan}.
In particular, we refer the reader to \cite{Taha} for a nonlinear Crank-Nicolson-type
finite difference method, to \cite{AlHarbi} for a linearized Crank-Nicolson finite difference method, and
to \cite{Raslan} for a Crank-Nicolson/spline collocation method. We note that none
of the above mentioned papers includes an error analysis of the methods proposed and numerically tested.
We would like also to stress that we are not aware of any other scientific work dealing with the error analysis
of a numerical method for the (SH) equation.
\par
We close this section by giving a brief overview of the paper.
In Section~\ref{Section2} we introduce notation and we prove 
a series of auxiliary results that we will often use later
in the analysis of the (MFDS) and the (FDM) approximations.
Section~\ref{Section3} is dedicated to  the construction and
the analysis of the (MFDS) approximations.
In Section~\ref{Section4} we show the well-posedness and the
convergence of the (FDM) approximations.
Finally, we expose results from numerical
experiments in Section~\ref{Section5}.
%
%
\section{Notation and Preliminaries}\label{Section2}
%
%
Let $\partial_h:\dspace\rightarrow\dspace$ be the discrete space derivative 
operator introduced by \eqref{discrete_D1a}; then, for $\ell\ge2$, we define,
recursively, an $\ell-$order discrete space derivative operator
$\partial^{\ell}_h:\dspace\rightarrow\dspace$,  by
$\partial_h^{\ell}=\partial_h\circ\partial^{\ell-1}_h$,
under the notation convention $\partial_h^1\equiv\partial_h$.
In addition, we define the shift operators $\sigma^{+}$,
$\sigma^{-}:\dspace\rightarrow\dspace$  by
$\sigma^{+}v_j:=v_{j+1}$ and $\sigma^{-}v_j:=v_{j-1}$ for $j\in{\mathbb
Z}$ and $v\in\dspace$, and another discrete space derivative operator
$\delta_h:\dspace\rightarrow\dspace$ by
\begin{equation*}
\delta_hv_j:=\tfrac{v_{j+1}-v_j}{h}\quad\forall\,j\in{\mathbb Z},
\quad\forall\,v\in\dspace.
\end{equation*}
\par
On $\dspace$ we define the discrete inner products
$(\cdot,\cdot)_{0,h}$ and $(\cdot,\cdot)_{1,h}$ by
$(v,z)_{0,h}:=h\sum_{j=1}^{\ssy J}v_j\,\overline{z_j}$ and
$(v,z)_{1,h}:=(v,z)_{0,h}+(\partial_hv,\partial_hz)_{0,h}$ for
$v,z\in\dspace$,
and we shall denote by $\|\cdot\|_{0,h}$ the norm corresponding to
the inner product $(\cdot,\cdot)_{0,h}$, i.e.
$\|v\|_{0,h}:=\sqrt{(v,v)_{0,h}}$ for
$v\in\dspace$. Also, we define on ${\mathfrak X}_h$
a discrete $L^{\infty}$-norm $|\cdot|_{\infty,h}$
by $|v|_{\infty,h}:=\max_{1\leq j\leq{\ssy J}}|v_j|$ for
$v\in\dspace$.
For $m\in{\mathbb N}$ we introduce a discrete $H^m$-seminorm
$|\cdot|_{m,h}$ by $|v|_{m,h}:=\|\partial_{h}^mv\|_{0,h}$
for $v\in\dspace$, a discrete $H^m$-norm $\|\cdot\|_{m,h}$ by
$\|v\|_{m,h}:=\left(\,\|v\|_{0,h}^2+\sum_{\ell=1}^{m}
|v|^2_{\ell,h}\,\right)^{\frac{1}{2}}$ for $v\in\dspace$,
a discrete $W^{m,\infty}$-seminorm
$|\cdot|_{m,\infty,h}$ by
$|v|_{m,\infty,h}:=|\partial_h^mv|_{\infty,h}$ for
$v\in\dspace$, and a discrete $W^{m,\infty}$-norm
$\|\cdot\|_{m,\infty,h}$, by
$\|v\|_{m,\infty,h}:=\max\left\{|v|_{\infty,h},
\max_{1\leq{\ell}\leq{m}}|v|_{\ell,\infty,h}\right\}$ for
$v\in\dspace$.
For given norm $\nu$ on $\dspace$, $v\in\dspace$ and $\varepsilon>0$,
we define the closed ball
\begin{equation*}
{\mathfrak B}_{\sf c}(v,\varepsilon;\nu):=\left\{w\in\dspace:\,\nu(v-w)\leq\varepsilon\right\}.
\end{equation*}
%
%
\par
Below, we provide a series of auxiliary results that we will 
often use in the rest of the paper.
%
%
%
%
\begin{lem}
For all $v$, $z\in\dspace$, we have
\begin{gather}
\sigma^{+}\sigma^{-}z=\sigma^{-}\sigma^{+}z=z,\label{NewEraM10}\\
\partial_h\Delta_hv=\Delta_h\partial_hv,\label{PROP_commut}\\
\partial_h{\mathcal A}_hv={\mathcal A}_h\partial_hv,\label{PROP_commut2}\\
%
%
\partial_h(v{\otimes}z)=\partial_hv{\otimes}\sigma^{+}z
+\sigma^{-}v\otimes\partial_hz.\label{PROD_DIFF}
\end{gather}
\end{lem}
%
%
%
\begin{proof}
The verification of the formulas above is straightforward.
\end{proof}
%
%
\begin{lem}\label{Lemma1}
For all $v,z\in\dspace$ it holds that
\begin{gather}
(\sigma^{+}v,z)_{0,h}=(v,\sigma^{-}z)_{0,h},\label{R_Base1}\\
(\sigma^{-}v,z)_{0,h}=(v,\sigma^{+}z)_{0,h},\label{R_Base2}\\
(\partial_hv,z)_{0,h}=-(v,\partial_hz)_{0,h},\label{PROP_I}\\
(\Delta_hv,z)_{0,h}=-(\delta_hv,\delta_hz)_{0,h}=(v,\Delta_hz)_{0,h},
\label{NewEra1}\\
(\Delta_hv,v)_h=-\|\delta_hv\|^2_{0,h},\label{PROP_II}\\
\mathrm{Re}(\partial_h\Delta_hv,v)_{0,h}=0.\label{PROP_III}
\end{gather}
\end{lem}
%
%
%
%
\begin{proof}
Let $v,z\in\dspace$. First, we establish \eqref{R_Base1}
proceeding as follows:
\begin{equation*}
\begin{split}
(\sigma^{+}v,z)_{0,h}=&\,h\,\sum_{j=1}^{\ssy J}v_{j+1}\,\overline{z_j}
=h\,\sum_{j=2}^{\ssy J+1}v_j\,\overline{z_{j-1}}
=h\,\sum_{j=1}^{\ssy J}v_j\,\overline{z_{j-1}}
=(v,\sigma^{-}z)_{0,h}.
\end{split}
\end{equation*}
Then, we apply \eqref{R_Base1} to obtain 
\begin{equation*}
\begin{split}
(\sigma^{-}v,z)_{0,h}=&\,\overline{(z,\sigma^{-}v)_{0,h}}
=\overline{(\sigma^{+}z,v)_{0,h}}
=(v,\sigma^{+}z)_{0,h}.\\
\end{split}
\end{equation*}
To obtain \eqref{PROP_I}, we combine
\eqref{R_Base1} and \eqref{R_Base2} as follows:
\begin{equation*}
(\partial_hv,z)_{0,h}=\tfrac{1}{2h}\,(\sigma^{+}v-\sigma^{-}v,z)_{0,h}
=\tfrac{1}{2h}\,(v,\sigma^{-}z-\sigma^{+}z)_{0,h}
=-(v,\partial_hz)_{0,h}.
\end{equation*}
Also, using \eqref{R_Base2}, we have
\begin{equation}\label{Dec2015_31a}
\begin{split}
(\Delta_hv,z)_{0,h}=&\,\tfrac{1}{h^2}(\sigma^{+}v-2\,v+\sigma^{-}v,z)_{0,h}\\
=&\,\tfrac{1}{h}\,\left[\,(\delta_hv,z)_{0,h}-(\sigma^{-}\delta_hv,z)_{0,h}\,\right]\\
=&\,\tfrac{1}{h}\,(\delta_hv,z-\sigma^{+}z)_{0,h}\\
=&\,-(\delta_hv,\delta_hz)_{0,h}\\
\end{split}
\end{equation}
which yields
\begin{equation}\label{Dec2015_31b}
(v,\Delta_hz)=\overline{(\Delta_hz,v)_{0,h}}
=-\overline{(\delta_hz,\delta_hv)_{0,h}}
=-(\delta_hv,\delta_hz)_{0,h}.
\end{equation}
Thus, \eqref{NewEra1} is a simple consequence of \eqref{Dec2015_31a} and
\eqref{Dec2015_31b}. Relation \eqref{PROP_II} follows easily from
\eqref{NewEra1} setting $z=v$. Finally, we use \eqref{PROP_commut}, \eqref{NewEra1} and
\eqref{PROP_I}, to have
\begin{equation*}
\begin{split}
(\partial_h\Delta_hv,v)_{0,h}=&(\Delta_h\partial_hv,v)_{0,h}
=(\partial_hv,\Delta_hv)_{0,h}
=-(v, \partial_h\Delta_hv)_{0,h}
=-{\overline{(\partial_h\Delta_hv,v)_{0,h}}},
\end{split}
\end{equation*}
which, obviously, yields \eqref{PROP_III}.
\end{proof}
%
%
%
%
\begin{lem}\label{Lemma2}
For all $v,z\in\dspace$ it holds that
\begin{gather}
|\partial_h(|z|^2)|_{\infty,h}\leq
\,2\,\,|z|_{\infty,h}\,\,|z|_{1,\infty,h},\label{BPROP_I}\\
%
%
\big\|\,|z|^2-|v|^2\,\big\|_{0,h}\leq
\,(|z|_{\infty,h}+|v|_{\infty,h})\,\,\|z-v\|_{0,h}.
\label{BPROP_III}
\end{gather}
\end{lem}
%
%
\begin{proof}
Let $v$, $z\in\dspace$. Observing that
\begin{equation*}
|\partial_h(|z|^2)|_{\infty,h}=\max_{1\leq{j}\leq{\ssy J}}
\left|{\mathrm{Re}}\left[\,\partial_hz_j\,{\overline{(z_{j+1}+z_{j-1})}}
\,\right]\,\right|
\end{equation*}
and
\begin{equation*}
\big\|\,|z|^2-|v|^2\,\big\|_{0,h}=\left[h\sum_{j=1}^{\ssy
J} \left|{\mathrm{Re}}\left[\,
(z_j-v_j)\,{\overline{(z_j+v_j)}}
\,\right]\,\right|^2\right]^{\frac{1}{2}}
\end{equation*}
\eqref{BPROP_I} and \eqref{BPROP_III} easily follow.
\end{proof}
%
%
%
%
%
%
%
%
%
%
%
\begin{lem}
Let ${\mathcal A}_h$ be the space average operator
defined by \eqref{fyrom1}. Then, for $w$, $z\in\dspace$,
it holds that
\begin{gather}
({\mathcal A}_hw,z)_{0,h}=(w,{\mathcal A}_hz)_{0,h},\label{PROP_IV}\\
\|\sigma^{+}z\|_{0,h}=\|z\|_{0,h},\label{NewEra3}\\
\|\sigma^{-}w\|_{0,h}=\|w\|_{0,h},\label{NewEra4}\\
\|{\mathcal A}_hz\|_{0,h}\leq\|z\|_{0,h},\label{PROP_V}\\
{\mathrm{Re}}\left(\,{\mathcal
A}_h(|z|^2\otimes\partial_hw),w\,\right)_{0,h}
=-\tfrac{1}{2}
\,\left(\,\partial_h(|z|^2),|w|^2\,\right)_{0,h}.\label{PROP_VII}
\end{gather}
\end{lem}
%
%
%
%
\begin{proof}
Let $w$, $z\in\dspace$. Using \eqref{R_Base1} and \eqref{R_Base2},
we obtain \eqref{PROP_IV} proceeding as follows
\begin{equation*}
({\mathcal A}_hw,z)_{0,h}
=\tfrac{1}{2}\,(\sigma^{+}w+\sigma^{-}w,z)_{0,h}
=\tfrac{1}{2}\,(w,\sigma^{-}z+\sigma^{+}z)_{0,h}\\
=(w,{\mathcal A}_hz)_{0,h}.
\end{equation*}
Combining \eqref{R_Base1} and \eqref{NewEraM10}, we have
\begin{equation*}
(\sigma^{+}z,\sigma^{+}z)_{0,h}
=(z,\sigma^{-}\sigma^{+}z)_{0,h}=(z,z)_{0,h},
\end{equation*}
which, obviously yields \eqref{NewEra3}. To arrive at
\eqref{NewEra4}, we use \eqref{NewEra3} and
\eqref{NewEraM10} as follows
\begin{equation*}
\|\sigma^{-}w\|_{0,h}=\|\sigma^{+}\sigma^{-}w\|_{0,h}=\|w\|_{0,h}.
\end{equation*}
Also, we apply \eqref{NewEra3} and \eqref{NewEra4} to get
\begin{equation*}
\begin{split}
\|{\mathcal A}_hz\|_{0,h}=&\,\tfrac{1}{2}\,\|\sigma^{+}z+\sigma^{-}z\|_{0,h}\\
\leq&\,\tfrac{1}{2}\,\left(\,\|\sigma^{+}z\|_{0,h}+\|\sigma^{-}z\|_{0,h}\right)\\
\leq&\,\|z\|_{0,h},\\
\end{split}
\end{equation*}
which is \eqref{PROP_V}. Finally, we use \eqref{PROP_IV} and \eqref{PROP_I} to obtain
\begin{equation*}
\begin{split}
{\mathrm{Re}}\left(\,{\mathcal A}_h(|z|^2\otimes
\partial_hw),w\,\right)_{0,h}=&\,{\mathrm{Re}}\left(\,|z|^2\otimes
\partial_hw,{\mathcal A}_hw\,\right)_{0,h}\\
=&\,\tfrac{h}{2}\,\sum_{j=1}^{\ssy J}|z_j|^2
\,\tfrac{|w_{j+1}|^2-|w_{j-1}|^2}{2h}\\
=&\tfrac{1}{2}\,\left(|z|^2,\partial_h\left(|w|^2\right)\right)_{0,h}\\
=&\,-\tfrac{1}{2}\,\left(\,\partial_h(|z|^2),|w|^2\,\right)_{0,h}\\ 
\end{split}
\end{equation*}
which establishes \eqref{PROP_VII}.
%
%
%
%
\end{proof}
%
%
%
%
\begin{lem}\label{Sobemb1}
Let $J\geq 3$. Then, we have that
\begin{equation}\label{periodicSobo}
|\psi|_{\infty,h}\leq\,\tfrac{\sqrt{9+16\,{\sf L}^2}}{\sqrt{\sf L}}\,\,\|\psi\|_{1,h}
\quad\forall\,\psi\in\dspace.
\end{equation}
\end{lem}
%
%
%
%
\begin{proof}
Let $\psi\in\dspace$. It is easily seen that there exists $m\in\{J+1,\dots,2J\}$
such that $|\psi_m|=|\psi|_{\infty,h}$. Then, we consider the following cases:
\par\noindent\vskip0.2truecm\par
{\sc Case 1}: $m$ is odd, i.e. there exists ${\widetilde m}\in{\mathbb N}$ such that
$m=2\,{\widetilde m}+1$.
\par\noindent
Let $A:=\left\{\ell\in{\mathbb N}:\, \ell\leq J,\,\,\, \ell\equiv1 \mod 2\right\}$.
Then, for $\kappa\in A$ there exists $\rho(\kappa)\in{\mathbb N}$ such that
$\kappa=2\rho(\kappa)-1$. Thus, we have
\begin{equation*}
\begin{split}
|\psi_m|=&\,\left|\,\psi_{\kappa}
+2\,h\,\sum_{\ell=\rho(\kappa)}^{{\widetilde m}}\partial_h\psi_{2\ell}\right|\\
\leq&\,|\psi_{\kappa}|+2\,h\,\sum_{\ell=2\rho(\kappa)}^{2{\widetilde m}}|\partial_h\psi_{\ell}|
\quad\forall\,\kappa\in A,\\
\end{split}
\end{equation*}
which yields
\begin{equation}\label{discSob1}
\begin{split}
{\mathrm{card}}(A)\,|\psi|_{\infty,h}\leq&\,\sum_{\kappa\in
A}|\psi_{\kappa}|+2h\,\sum_{\kappa\in A}\left(\,
\sum_{\ell=\kappa+1}^{m-1}|\partial_{h}\psi_{\ell}|\,\right)\\
\leq&\,\sum_{\kappa\in A}|\psi_{\kappa}|+2h\,\mathrm{card}(A)
\sum_{\ell=2}^{m-1}|\partial_{h}\psi_{\ell}|.\\
\end{split}
\end{equation}
Observing  that
\begin{equation*}
{\mathrm{card}}(A)=\left\{
\begin{aligned}
&\tfrac{J}{2}\hskip1.0truecm\text{\rm if}\ \ J \ \ \text{\rm is even}\\
&\tfrac{J+1}{2}\hskip0.65truecm\text{\rm if}\ \ J \ \ \text{\rm is odd}\\
\end{aligned}
\right.,
\end{equation*}
and using \eqref{discSob1} along with the Cauchy-Schwarz inequality,
we obtain
\begin{equation}\label{Sobo1}
\begin{split}
|\psi|_{\infty,h}\leq&\,\tfrac{2}{J}\,\sum_{\ell\in A}
|\psi_{\ell}|+2h\,\sum_{\ell=1}^{\ssy 2J}|\partial_h\psi_{\ell}|\\
\leq&\,\tfrac{2}{\sf L}\,h\,\sum_{\ell=1}^{\ssy J}|\psi_{\ell}|
+4\,h\,\sum_{\ell=1}^{\ssy J}|\partial_h\psi_{\ell}|\\
\leq&\,\tfrac{2}{\sqrt{\sf L}}\,\|\psi\|_{0,h}+4\,\sqrt{\sf L}\,|\psi|_{1,h}\\
\leq&\,2\,\tfrac{\sqrt{1+4\,{\sf L}^2}}{\sqrt{\sf L}}\,\,\|\psi\|_{1,h}.\\
\end{split}
\end{equation}
\par\noindent\smallskip\par
{\sc Case 2}: $m$ is even, i.e. there exists
$\widetilde{m}\in{\mathbb N}$ such that $m=2\,{\widetilde m}$.
\par\noindent
Let $B:=\left\{\ell\in{\mathbb N}:\, \ell\leq J,\,\,\, \ell\equiv0 \mod 2\right\}$.
Then, for $\kappa\in B$ there exists  $\rho(\kappa)\in{\mathbb N}$ such that $\kappa=2\rho(\kappa)$.
First we observe that
\begin{equation*}
\begin{split}
|\psi_m|=&\,\left|\psi_{\kappa}
+2h\sum_{\ell=\rho(\kappa)}^{{\widetilde{m}}-1}
\partial_h\psi_{2\ell+1}\right|\\
\leq&\,|\psi_{\kappa}|+2\,h\,\sum_{\ell=2\rho(\kappa)+1}^{2{\widetilde{m}}-1}
|\partial_h\psi_{\ell}|\quad\forall\,\kappa\in B.\\
\end{split}
\end{equation*}
Then, we sum over $\kappa\in B$ to get
\begin{equation}\label{DiscSob5}
\begin{split}
\mathrm{card}(B)\,|\psi|_{\infty,h}\leq&\,
\sum_{\kappa\in B}|\psi_{\kappa}|+2\,h\,\sum_{\kappa\in B}
\left(\,\sum_{\ell=\kappa+1}^{m-1}|\partial_h\psi_{\ell}|\,\right)\\
\leq&\,
\sum_{\kappa\in B}|\psi_{\kappa}|+2\,h\,\mathrm{card}(B)
\sum_{\ell=3}^{m-1}|\partial_h\psi_{\ell}|.\\
\end{split}
\end{equation}
Observing that
\begin{equation*}
\mathrm{card}(B)=\left\{
\begin{aligned}
&\tfrac{J}{2}\hskip1.0truecm\mbox{if $J$ is even}\\
&\tfrac{J-1}{2}\hskip0.65truecm\mbox{if $J$ is odd}\\
\end{aligned}
\right.,
\end{equation*}
\eqref{DiscSob5} yields
\begin{equation}\label{Sobo2}
\begin{split}
|\psi|_{\infty,h}\leq&\,\tfrac{2}{J-1}\,\sum_{\kappa\in B}|\psi_{\kappa}|
+2\,h\,\sum_{\ell=1}^{\ssy 2J}|\partial_h\psi_{\ell}|\\
\leq&\,\tfrac{3}{\sf L}\,h\,\sum_{\ell=1}^{\ssy J}|\psi_{\ell}|
+4\,h\,\sum_{\ell=1}^{\ssy J}|\partial_h\psi_{\ell}|\\\
\leq&\,\tfrac{3}{\sqrt{\sf L}}\,\|\psi\|_{0,h}+4\,\sqrt{{\sf L}}\,|\psi|_{1,h}\\\
\leq&\,\tfrac{\sqrt{9+16\,{\sf L}^2}}{\sqrt{\sf L}}\,\|\psi\|_{1,h}.\\
\end{split}
\end{equation}
\par
The desired inequality \eqref{periodicSobo} is a simple consequence of \eqref{Sobo1} and \eqref{Sobo2}.
\end{proof}
%
%
%
\begin{lem}
The following discrete inverse inequality holds
\begin{equation}\label{maxL2}
|\psi|_{\infty,h}\leq\,h^{-\frac{1}{2}}\,\|\psi\|_{0,h}
\quad\forall\,\psi\in\dspace.
\end{equation}
\end{lem}
%
%
%
%
\begin{proof}
Let $\psi\in\dspace$ and $m\in\{1,\dots,J\}$ such that $|\psi|_{\infty,h}=|\psi_m|$. Then,
we have
\begin{equation*}
\|\psi\|_{0,h}^2=h\sum_{\ell=1}^{\ssy J}|\psi_{\ell}|^2\ge\,h\,|\psi_m|^2,
\end{equation*}
which easily yields \eqref{maxL2}.
\end{proof}
\section{A Modified Finite Difference Scheme}\label{Section3}
%
We will carry out the convergence analysis of the proposed (FDM)
by investigating the convergence of a properly defined Modified Finite Difference
Scheme (MFDS) that derives non-computable finite difference approximations of
the exact solution $\phi$ to the problem \eqref{HirotaPeriodic}-\eqref{HirotaPeriodic_b}
(cf. \cite{Georgios1}).
In particular, we will construct the (MFDS) using an operational mollification of the nonlinear
terms in (FDM), which is based on a given real parameter $\lambda>0$ and the norm
$\|\cdot\|_{1,\infty,h}$ on $\dspace$. The goal of this construction is to provide the
(MFDS) with the following key property:
`when the (MFDS) approximations  have $\|\cdot\|_{1,\infty,h}-$distance from the exact
solution to the problem lower than
$\lambda$, then they are also (FDM) approximations'.
%
\subsection{Constructing an Operational Mollification}\label{section31}
%
\par
For $\lambda>0$, let $\xi(\lambda;\cdot):{\mathbb
R}\rightarrow[0,1]$ be a continuous function defined by
\begin{equation}\label{xi_function}
\xi(\lambda;x):=\left\{
\begin{aligned}
&1,\hskip1.0truecm
\mbox{if}\ \ x\leq \lambda,\\
&\tfrac{2\lambda-x}{\lambda},\hskip0.4truecm
\mbox{if}\ \ x\in (\lambda,2\lambda],\\
&0,\hskip1.0truecm\mbox{if}\ \ x> 2\lambda,\\
\end{aligned}
\right.\quad\forall\,x\in{\mathbb R}.
\end{equation}
%
%
Then, for $\lambda>0$ and $t\in[0,T]$, we construct an operational mollifier
${\mathfrak m}(\lambda,t;\cdot):\dspace\rightarrow\dspace$ by
\begin{equation}\label{G_function}
\begin{split}
{\mathfrak m}(\lambda,t;w):=&\,w\,
\,\xi\big(\lambda;\|w-\Lambda_h(\phi(t,\cdot))\|_{1,\infty,h}\big)\\
&\quad
+\Lambda_h(\phi(t,\cdot))\,\left[1-\xi\big(\lambda;
\|w-\Lambda_h(\phi(t,\cdot))\|_{1,\infty,h}\big)\,\right]\quad\forall\,w\in\dspace,\\
\end{split}
\end{equation}
where $\phi$ is the solution to the problem
\eqref{HirotaPeriodic}.
%
%
\par
In the lemmas below, we establish some usefull properties of the map
${\mathfrak m}(\lambda,t;\cdot)$.
%
%
\begin{lem}
It holds that
\begin{equation}\label{mol00}
{\mathfrak m}(\lambda,t;v)=v\quad\forall\,v
\in{\mathfrak B}_{\sf c}\left(\Lambda_h(\phi(t,\cdot)),\lambda;\|\cdot\|_{1,\infty,h}\right),
\quad\forall\,\lambda>0,\quad\forall\,t\in[0,T],
\end{equation}
%
%
and
\begin{equation}\label{mol2}
\|{\mathfrak m}(\lambda,t;w)\|_{1,\infty,h}<3\lambda
\quad\forall\,w\in\dspace,\quad\forall\,\lambda\ge\lambda_{\star},
\quad\forall\,t\in[0,T],
\end{equation}
where $\lambda_{\star}:=\max\limits_{0\leq{\ell}\leq{1}}
\,\left(\max_{\ssy{[0,1]\times[0,T]}}|\partial_x^{\ell}\phi|\right)$.
\end{lem}
%
%
%
%
\begin{proof}
Let $t\in[0,T]$, $\lambda>0$ and
$v\in{\mathfrak B}_{\sf c}\left(\Lambda_h(\phi(t,\cdot)),\lambda;\|\cdot\|_{1,\infty,h}\right)$.
Then, it holds that
\begin{equation*}
\|v-\Lambda_h(\phi(t,\cdot))\|_{1,\infty,h}\leq\lambda
\end{equation*}
which, along with \eqref{xi_function}, yields
\begin{equation}\label{traka_mol}
\xi(\lambda; \|v-\Lambda_h(\phi(t,\cdot))\|_{1,\infty,h})=1.
\end{equation}
The equality \eqref{mol00} follows easily combining \eqref{G_function}
and \eqref{traka_mol}.
\par
Now, let us assume that $\lambda\ge\lambda_{\star}$. Then, we use
\eqref{G_function} to get
\begin{equation}\label{mol4}
\begin{split}
\|{\mathfrak m}(\lambda,t;w)\|_{1,\infty,h}\leq&\,
\|w\|_{1,\infty,h}\,\,\,\xi\big(\lambda;
\|w-\Lambda_h(\phi(t,\cdot))\|_{1,\infty,h}\big)\\
&\quad +\|\Lambda_h(\phi(t,\cdot))\|_{1,\infty,h}
\,\,\,\left[1-\xi\left(\lambda;
\|w-\Lambda_h(\phi(t,\cdot))\|_{1,\infty,h}\right)\right].\\
\end{split}
\end{equation}
If $\|w-\Lambda_h(\phi(t,\cdot))\|_{1,\infty,h}\ge2\lambda$, then
\eqref{xi_function} and \eqref{mol4} yield
\begin{equation}\label{mol5}
\begin{split}
\|{\mathfrak m}(\lambda,t;w)\|_{1,\infty,h}
\leq&\,\|\Lambda_h(\phi(t,\cdot))\|_{1,\infty,h}\\
\leq&\,\lambda_{\star}\\
\leq&\,\lambda.\\
\end{split}
\end{equation}
If $\|w-\Lambda_h(\phi(t,\cdot))\|_{1,\infty,h}<2\lambda$, then
\begin{equation}\label{mol6}
\begin{split}
\|{\mathfrak m}(\lambda,t;w)\|_{1,\infty,h}
\leq&\,\max\left\{\,\|w\|_{1,\infty,h},
\|\Lambda_h(\phi(t,\cdot))\|_{1,\infty,h}\,\right\}\\
\leq&\,\max\left\{\,\|w-\Lambda_h(\phi(t,\cdot))\|_{1,\infty,h}
+\|\Lambda_h(\phi(t,\cdot))\|_{1,\infty,h},
\|\Lambda_h(\phi(t,\cdot))\|_{1,\infty,h}\,\right\}\\
<&\,2\,\lambda+\|\Lambda_h(\phi(t,\cdot))\|_{1,\infty,h}\\
<&\,2\lambda+\lambda_{\star}\\
<&\,3\lambda.\\
\end{split}
\end{equation}
Thus, \eqref{mol2} follows easily from \eqref{mol5} and \eqref{mol6}.
\end{proof}
%
%
%
%
\begin{lem}
Let $\nu$ be a seminorm on $\dspace$. Then,
it holds that
\begin{equation}\label{G_Lip}
\nu\left(\,{\mathfrak m}(\lambda,t;w)-\Lambda_h(\phi(t,\cdot))\,\right)
\leq\,\nu\left(\,w-\Lambda_h(\phi(t,\cdot))\,\right)\quad\forall\,w\in\dspace,
\quad\forall\,\lambda>0,\quad\forall\,t\in[0,T].
\end{equation}
\end{lem}
%
%
%
%
\begin{proof}
Let $t\in[0,T]$, $\lambda>0$ and $w\in\dspace$. Then, 
\eqref{G_function} and \eqref{xi_function} yield
\begin{equation*}
\begin{split}
\nu\left(\,{\mathfrak m}(\lambda,t;w)-\Lambda_h(\phi(t,\cdot))\,\right)
=&\,\nu\left(\,w-\Lambda_h(\phi(t,\cdot))\,\right)\,\,\,
\xi\left(\lambda;\|w-\Lambda_h(\phi(t,\cdot))\|_{1,\infty,h}\right)\\
\leq&\,\nu\left(\,w-\Lambda_h(\phi(t,\cdot))\,\right).\\
\end{split}
\end{equation*}
\end{proof}
%
%
%
%
%
%
\subsection{The Modified Finite Difference Scheme (MFDS)}\label{Here_Is_MFD}
%
Here, we introduce a modified finite difference scheme which, for
$\lambda>0$, derives non-computable approximations
$(S^n(\lambda))_{n=0}^{\ssy N}\subset\dspace$ 
of the solution $\phi$ to \eqref{HirotaPeriodic}-\eqref{HirotaPeriodic_b},
following the steps below:
\par
{\tt Step A}: First, set
\begin{equation}\label{modfidifin}
S^0(\lambda):=\phi^0.
\end{equation}
\par
{\tt Step B}: Find $S^{1}(\lambda)\in\dspace$ such that
\begin{equation}\label{modscheme1a}
\begin{split}
\tfrac{S^1(\lambda)-S^0(\lambda)}{\tau}= &\, {\rm
i}\,\rho\,\Delta_h\left(\tfrac{S^1(\lambda)+S^0(\lambda)}{2}\right)
-\sigma\,\partial_h\Delta_h\left(\tfrac{S^1(\lambda)+S^0(\lambda)}{2}\right)\\
&\,+{\rm i}\,\delta\,\left[\,\left|{\mathfrak m}\left(\lambda,t_0;S^0(\lambda)\right)\right|^2 \otimes
\left(\tfrac{S^1(\lambda)+S^0(\lambda)}{2}\right)\,\right]\\
&\,-3\,\alpha\, {\mathcal A}_h\left[\,\left|{\mathfrak m}\left(\lambda,t_0;S^0(\lambda)\right)\right|^2
\otimes\partial_h\left(\tfrac{S^1(\lambda)+S^0(\lambda)}{2}\right)\,\right]+F^{\half}.\\
\end{split}
\end{equation}
\par
{\tt Step C}: For $n=1,\dots,N-1$, find $S^{n+1}(\lambda)\in\dspace$ such
that
\begin{equation}\label{modscheme1}
\begin{split}
\tfrac{S^{n+1}(\lambda)-S^{n-1}(\lambda)}{2\tau}=&\,
{\rm i}\,\rho\,\Delta_h\left(\tfrac{S^{n+1}(\lambda)+S^{n-1}(\lambda)}{2}\right)
-\sigma\,\partial_h\Delta_h\left(\tfrac{S^{n+1}(\lambda)+S^{n-1}(\lambda)}{2}\right)\\
&\,+{\rm i}\,\delta\,\left[\,\left|{\mathfrak m}\left(\lambda,t_n;S^n(\lambda)\right)\right|^2
\otimes\left(\tfrac{S^{n+1}(\lambda)+S^{n-1}(\lambda)}{2}\right)\,\right]\\
&\,-3\,\alpha
\,{\mathcal A}_h\left[\,\left|{\mathfrak m}\left(\lambda,t_n;S^n(\lambda)\right)\right|^2
\otimes\partial_h\left(\tfrac{S^{n+1}(\lambda)+S^{n-1}(\lambda)}{2}\right)\right]+F^n.\\
\end{split}
\end{equation}
%
%
\subsection{Existence and uniqueness of the (MFDS) approximations}
%
Below we show that, if $\tau$ is small enough, then the (MFDS)
approximations are well-defined.
%
%
%
%
%
%
%
%
%
\begin{proposition}\label{ModExist}
Let $\lambda_{\star}:=\max_{0\leq{\ell}\leq1}
\,\left(\,\max_{\ssy{[0,1]\times[0,T]}}|\partial_x^{\ell}\phi|\right)$
and $\lambda\ge\lambda_{\star}$.
Then, there exist a constant ${\sf C}_1>0$, depending only on
$\alpha$, such that: if $\tau\,{\sf C}_1\,\lambda^2<1$,
then modified finite difference approximations \eqref{modfidifin}-\eqref{modscheme1}
are well-defined.
\end{proposition}
%
%
%
%
\begin{proof}
Let $t\in[0,T]$, $\zeta>0$ and $\chi\in\dspace$.
Then, we define linears operators
${\sf Q}(t,\zeta,\chi):\dspace\rightarrow\dspace$ and
${\sf T}(t,\zeta,\chi):\dspace\rightarrow\dspace$ by
\begin{equation*}\label{Odefinition}
\begin{split}
{\sf T}(t,\zeta,\chi)v:=&\,2\,v+{\sf Q}(t,\zeta,\chi)v
\quad\forall\,v\in\dspace,\\
\end{split}
\end{equation*}
where
\begin{equation*}
\begin{split}
{\sf Q}(t,\zeta,\chi)v:=&\,-{\rm i}\,\rho\,\tau\,\zeta\,\Delta_hv 
+\sigma\,\tau\,\zeta\,\partial_h\Delta_hv
-{\rm i}\,\delta\,\tau\,\zeta\,\left[\,
\left|{\mathfrak m}(\lambda,t;\chi)\right|^2\otimes v\,\right]\\
&\hskip1.0truecm
+3\,\alpha\,\tau\,\zeta\,{\mathcal A}_h\left[\,
\left|{\mathfrak m}(\lambda,t;\chi)\right|^2\otimes\partial_hv\,\right]
\quad\forall\,v\in\dspace.\\
\end{split}
\end{equation*}
Using \eqref{PROP_II}, \eqref{PROP_III}, \eqref{PROP_VII}, the
Cauchy-Schwarz inequality, \eqref{BPROP_I} and \eqref{mol2}, we
have
\begin{equation}\label{doveE1}
\begin{split}
{\mathrm{Re}}\big({\sf T}(t,\zeta,\chi)v,v\big)_{0,h}=&\,2\,\|v\|_{0,h}^2
-\tfrac{3}{2}\,\alpha\,\tau\,\zeta\,{\mathrm{Re}}\left(
\partial_h(|{\mathfrak m}(\lambda,t;\chi)|^2),|v|^2\right)_{0,h}\\
\ge&\,2\,\|v\|_{0,h}^2
-\tfrac{3}{2}\,|\alpha|\,\tau\,\zeta\,\|v\|_{0,h}^2
\,\,\,\left|\partial_h\left(|{\mathfrak m}(\lambda,t;\chi)|^2\right)\right|_{\infty,h}\\
\ge&\,2\,\|v\|_{0,h}^2
-3\,|\alpha|\,\tau\,\zeta\,\|v\|_{0,h}^2
\,\big\|{\mathfrak m}(\lambda,t;\chi)\big\|_{1,\infty,h}^2\\
\ge&\,2\,\|v\|_{0,h}^2\,\left(\,1-\tau\,\zeta\,\tfrac{27}{2}\,|\alpha|\,\lambda^2\,\right)
\quad\forall\,v\in\dspace.\\
\end{split}
\end{equation}
Let $\tfrac{27}{2}\,|\alpha|\,\lambda^2\,\tau\,\zeta<1$ and $v\in{\rm Ker}({\sf T}(t,\zeta,\chi))$.
Then, ${\mathrm{Re}}({\sf T}(t,\zeta,\chi)v,v)_{0,h}=0$,
which, along with \eqref{doveE1}, yields $v=0$.
Thus, ${\rm Ker}({\sf T}(t,\zeta,\chi))=\{0\}$ and
${\sf T}(t,\zeta,\chi)$ is invertible, since $\dspace$ has finite dimension.
\par
Set ${\sf C}_1:=27\,|\alpha|$ and require $\tau\,{\sf C}_1\,\lambda^2<1$. 
Then, according to the discussion above, the element
$S^1(\lambda)={\sf T}^{-1}(t_0,1,\phi^0)\psi^1$
is the solution to \eqref{modscheme1a}, with
$\psi^1:=2\,\phi^0+2\,\tau\,F^{\half}-{\sf Q}(t_0,1,\phi^0)\phi^0$.
%
%
Let $\kappa\in\{1,\dots,N-1\}$,
$S^1(\lambda)$ be the solution to the linear system \eqref{modscheme1a}
and $(S^m(\lambda))_{m=2}^{\kappa}$ be the solutions to the
linear system \eqref{modscheme1} for $n=1,\dots,\kappa-1$, respectively.
Then, the element
$S^{\kappa+1}(\lambda)={\sf T}^{-1}(t_{\kappa},2,S^{\kappa}(\lambda))\psi^{\kappa}$
is a solution of \eqref{modscheme1} for $n=\kappa$, with
$$\psi^{\kappa}:=2\,S^{\kappa-1}(\lambda)+4\,\tau\,F^{\kappa}
-{\sf Q}(t_{\kappa},2,S^{\kappa}(\lambda))S^{\kappa-1}(\lambda).$$
%
%
%
\end{proof}
%
%
\subsection{Consistency of  (MFDS) and (FDM) approximations}
%
%
Let $(\eta^n)_{n=0}^{\ssy N-1}\subset\dspace$ be defined by
\begin{equation}\label{consistency_equations0}
\begin{split}
\tfrac{\phi^1-\phi^0}{\tau}=&\, {\rm
i}\,\rho\,\Delta_h\left(\tfrac{\phi^1+\phi^0}{2}\right)
-3\,\alpha\,{\mathcal A}_h\left[\,|\phi^0|^2
\otimes\partial_h\left(\tfrac{\phi^1+\phi^0}{2}\right)\,\right]\\
&\,-\sigma\,\partial_h\Delta_h\left(\tfrac{\phi^1+\phi^0}{2}\right)
+{\rm i}\,\delta\,\left[\,|\phi^0|^2
\otimes\left(\tfrac{\phi^1+\phi^0}{2}\right)\,\right]+F^{\half}+\eta^0
\end{split}
\end{equation}
and
\begin{equation}\label{consistency_equations}
\begin{split}
\tfrac{\phi^{n+1}-\phi^{n-1}}{2\,\tau}=&\,
{\rm i}\,\rho\,\Delta_h\left(\tfrac{\phi^{n+1}+\phi^{n-1}}{2}\right)
-3\,\alpha\,{\mathcal A}_h\left[\,|\phi^n|^2
\otimes\partial_h\left(\tfrac{\phi^{n+1}+\phi^{n-1}}{2}\right)\,\right]\\
&\,-\sigma\,\partial_h\Delta_h\left(\tfrac{\phi^{n+1}+\phi^{n-1}}{2}\right)
+{\rm i}\,\delta\,\left[\,
|\phi^n|^2\otimes\left(\tfrac{\phi^{n+1}+\phi^{n-1}}{2}\right)\right]+F^n+\eta^n
\end{split}
\end{equation}
for $n=1,\dots,N-1$. Then, assuming enough space and time
regularity for the solution $\phi$ and using the Taylor formula,
we conclude that there exits positive real constants $C_1$
and $C_2$, which are independent of $\tau$ and $h$, such that:
\begin{gather}
\|\eta^0\|_{1,\infty,h}\leq\,C_1\,(h^2+\tau),\label{consistency_bound0}\\
\max_{1\leq{n}\leq{\ssy N-1}} 
\|\eta^n\|_{1,\infty,h}\leq\,C_2\,(h^2+\tau^2).\label{consistency_bound}
\end{gather}
Since the property \eqref{mol00} yields that $|\phi^n|^2=|{\mathfrak m}(\lambda,t_n;\phi^n)|^2$
for $n=0,\dots,N$, the consistency result described above for the (FDM) approximations
is, also, a consistency result for the (MFDS) approximations 
introduced in Section~\ref{Here_Is_MFD}.
%
\subsection{Convergence of the (MFDS) approximations}\label{Conv}
%
In the theorem below, we show convergence of  the (MFDS) approximations
in the discrete $H^1-$norm.
%
%
%
\begin{thm}\label{Conv_Of_Mod}
Let $\mu_{\star}:=\max_{0\leq{\ell}\leq{2}}
\left(\max_{\ssy{[0,1]\times[0,T]}}|\partial_x^{\ell}\phi|\right)$,
$\lambda_{{\sf c}}:=\mu_{\star}+1$ and ${\sf C}_1$
be the constant speficied in Proposition~\ref{ModExist}.
Also, we assume that $\tau\,{\sf C}_1\,\lambda_{{\sf c}}^2<1$
and denote by $({\sf Z}^{\ell})_{\ell=0}^{\ssy N}$
the modified finite difference approximations defined by
\eqref{modfidifin}--\eqref{modscheme1} for
$\lambda=\lambda_{{\sf c}}$, i.e.,
${\sf Z}^{\ell}=S^{\ell}(\lambda_{{\sf c}})$ for $\ell=0,\dots,N$. 
Then, there exist positive constants
${\sf C}_2$ and ${\sf C}_3\ge{\sf C}_1$, independent
of $\tau$ and $h$, such that: if $\tau\,{\sf C}_3\,\lambda_{\sf c}^2<1$, then
\begin{equation}\label{mainre}
\max_{0\leq{m}\leq{\ssy N}}\|\phi^m-{\sf Z}^m\|_{1,h}\leq\,{\sf C}_2\,(\tau^2+h^2).
\end{equation}
\end{thm}
%
%
%
\begin{proof}
To simplify the notation, we set ${\sf E}^{m}:=\phi^m-{\sf Z}^m$,
${\sf D}^m:=\partial_h{\sf E}^m$ and
${\mathfrak m}_m:={\mathfrak m}(\lambda_{{\sf c}},t_m;{\sf Z}^m)$ for
$m=0,\dots,N$. Also, we set $\ell(0):=1$, $\ell(n)=n+1$  for $n=1,\dots,N-1$,
$r(0)=0$, $r(n)=n-1$ for $n=1,\dots,N-1$. We note that, since
$\lambda_{\star}<\lambda_{\sf c}$ and
$\tau\,{\sf C}_1\,\lambda_{{\sf c}}^2<1$,
Proposition~\ref{ModExist} yields the existence and uniqueness of
$({\sf Z}^m)_{m=1}^{\ssy N}$.
In the sequel,  we will use the symbol $C$ to denote a generic constant that is
independent of $\tau$ and $h$, and may changes value from one line to the other.
%
%
\par\noindent
\vskip0.2truecm
\par
{\sc Step 1}. We subtract \eqref{modscheme1a} from \eqref{consistency_equations0}
and \eqref{modscheme1} from \eqref{consistency_equations}, to obtain the following
error equations:
\begin{gather}
\tfrac{{\sf E}^1-{\sf E}^0}{\tau}
={\rm i}\,\rho\,\Delta_h\left(\tfrac{{\sf E}^1+{\sf E}^0}{2}\right)
-\sigma\,\partial_{h}\Delta_h\left(\tfrac{{\sf E}^1+{\sf E}^0}{2}\right)
+\sum_{\kappa=1}^5{\sf A}_{\kappa,0},\label{Error_Equations_a}\\
\tfrac{{\sf E}^{n+1}-{\sf E}^{n-1}}{2\,\tau}=
{\rm i}\,\rho\,\Delta_h\left(\tfrac{{\sf E}^{n+1}+{\sf E}^{n-1}}{2}\right)
-\sigma\,\partial_h\Delta_h\left(\tfrac{{\sf E}^{n+1}+{\sf E}^{n-1}}{2}\right)
+\sum_{\kappa=1}^5{\sf A}_{\kappa,n},
\quad n=1,\dots,N-1,\label{Error_Equations}
\end{gather}
where
%
%
%
\begin{equation*}
\begin{split}
{\sf A}_{1,n}:=&\,-\tfrac{3\alpha}{2}\,{\mathcal A}_h\left[\,
\left(\, |\phi^n|^2-|{\mathfrak m}_n|^2
\,\right)\otimes\partial_h\left(\phi^{\ell(n)}+\phi^{r(n)}\right)\,\right],\\
{\sf A}_{2,n}:=&\,-\tfrac{3\alpha}{2}\,{\mathcal A}_h\left[\,|{\mathfrak m}_n|^2
\otimes\partial_h\left({\sf E}^{\ell(n)}+{\sf E}^{r(n)}\right)\right],\\
{\sf A}_{3,n}:=&\,{\rm i}\,\tfrac{\delta}{2}\,
\left[\,(|\phi^{n}|^2-|{\mathfrak m}_n|^2)
\otimes\left(\phi^{\ell(n)}+\phi^{r(n)}\right)\right],\\
{\sf A}_{4,n}:=&\,{\rm i}\,\tfrac{\delta}{2}\,\left[\,|{\mathfrak m}_n|^2
\otimes\left({\sf E}^{\ell(n)}+{\sf E}^{r(n)}\right)\,\right],\\
{\sf A}_{5,n}:=&\,\eta^n.\\
\end{split}
\end{equation*}
%
%
\par\noindent
\vskip0.2truecm
\par
{\sc Step 2}. We take the inner product $(\cdot,\cdot)_{0,h}$ of
\eqref{Error_Equations_a} with $({\sf E}^1+{\sf E}^0)$
and of \eqref{Error_Equations} with $({\sf E}^{n+1}+{\sf E}^{n-1})$.
Then, we keep the real part of the obtained relation and use \eqref{PROP_II}
and \eqref{PROP_III}, to have
%
%
%
\begin{gather}
\|{\sf E}^1\|_{0,h}^2-\|{\sf E}^0\|_{0,h}^2=
\tau\,\sum_{\kappa=1}^5a_{\kappa,0},\label{kwnos1_a}\\
\|{\sf E}^{n+1}\|_{0,h}^2-\|{\sf E}^{n-1}\|_{0,h}^2=
2\,\tau\,\,\sum_{\kappa=1}^5a_{\kappa,n},
\quad n=1,\dots,N-1,\label{kwnos1}
\end{gather}
where
\begin{equation*}
a_{\kappa,n}:={\mathrm{Re}}\big({\sf A}_{\kappa,n},{\sf E}^{\ell(n)}
+{\sf E}^{r(n)}\big)_{0,h}.
\end{equation*}
\par
Let $n\in\{0,\dots,N-1\}$.  First, we observe that
\begin{gather}
a_{4,n}=0,\label{kwnos2}\\
a_{5,n}\leq\,\|\eta^n\|_{0,h}
\,\left[\,\|{\sf E}^{\ell(n)}\|_{0,h}
+\|{\sf E}^{r(n)}\|_{0,h}\,\right].\label{kwnos5}
\end{gather}
Next, we apply the Cauchy-Schwarz inequality, \eqref{PROP_V} and
\eqref{BPROP_III} to get
\begin{equation*}
\begin{split}
a_{1,n}+a_{3,n}\leq&\,
\,\left(\tfrac{3|\alpha|}{2}+\tfrac{|\delta|}{2}\right)
\,\mu_{\star}\,\left\||\phi^n|^2-|{\mathfrak m}_n|^2\right\|_{0,h}
\,\|{\sf E}^{\ell(n)}+{\sf E}^{r(n)}\|_{0,h}\\
\leq&\,C\,\mu_{\star}\,\left(|\phi^n|_{\infty,h}+|{\mathfrak m}_n|_{\infty,h}\right)
\,\|\phi^n-{\mathfrak m}_n\|_{0,h}
\,\|{\sf E}^{\ell(n)}+{\sf E}^{r(n)}\|_{0,h},\\
\end{split}
\end{equation*}
which, along with \eqref{mol2} and \eqref{G_Lip}
(with $\nu(\cdot)=\|\cdot\|_{0,h}$), yields
\begin{equation}\label{kwnos6}
\begin{split}
a_{1,n}+a_{3,n}\leq&\,C\,\mu_{\star}\,(\mu_{\star}+3\,\lambda_{\sf c})
\,\|{\sf E}^n\|_{0,h}
\,\|{\sf E}^{\ell(n)}+{\sf E}^{r(n)}\|_{0,h}\\
\leq&\,C\,\lambda_{\sf c}^2\,\|{\sf E}^n\|_{0,h}
\,\left[\,\|{\sf E}^{\ell(n)}\|_{0,h}
+\|{\sf E}^{r(n)}\|_{0,h}\,\right].\\
\end{split}
\end{equation}
Also, we use \eqref{PROP_VII}, \eqref{BPROP_I} and
\eqref{mol2} to obtain
\begin{equation}\label{kwnos4}
\begin{split}
a_{2,n}\leq&\,\tfrac{3|\alpha|}{4}\,
\big|\,\big(\partial_h(|{\mathfrak m}_n|^2),|{\sf E}^{\ell(n)}+{\sf E}^{r(n)}|^2
\big)_{0,h}\,\big|\\
\leq&\,C\,\big|\partial_h(|{\mathfrak m}_n|^2)\big|_{\infty,h}
\,\,\,\|{\sf E}^{\ell(n)}+{\sf E}^{r(n)}\|_{0,h}^2\\
\leq&\,C\,\|{\mathfrak m}_n\|_{1,\infty,h}^2
\,\left[\,\|{\sf E}^{\ell(n)}\|_{0,h}
+\|{\sf E}^{r(n)}\|_{0,h}\right]^2\\
\leq&\,C\,\lambda_{\sf c}^2\,\left[\,\|{\sf E}^{\ell(n)}\|_{0,h}
+\|{\sf E}^{r(n)}\|_{0,h}\right]^2.\\
\end{split}
\end{equation}
\par
Letting $\nu_{\ssy{\sf E}}^m:=\|{\sf E}^m\|_{0,h}+\|{\sf E}^{m-1}\|_{0,h}$
for $m=1,\dots,N$, obseving that ${\sf E}^0=0$ and combining
\eqref{kwnos1_a}, \eqref{kwnos1}, \eqref{kwnos2}, \eqref{kwnos5},
\eqref{kwnos6},  \eqref{kwnos4}, \eqref{consistency_bound0}
and \eqref{consistency_bound}, we conclude that there exist positive constants
$C_{\ssy{\sf E},1}$ and $C_{\ssy{\sf E},2}$ such that
\begin{gather}
\nu_{\ssy{\sf E}}^{1}\leq\,C_{\ssy{\sf E},1}\,\lambda_{\sf c}^2
\,\tau\,\nu_{\ssy{\sf E}}^{1}+C_1\,(\tau^2+\tau\,h^2),
\label{L2RECURSIVEa}\\
\begin{split}
\nu_{\ssy{\sf E}}^{m+1}\leq&\,\nu_{\ssy{\sf  E}}^m
+C_{\ssy{\sf E},2}\,\lambda_{\sf c}^2\,\tau\,
\,\left(\nu_{\ssy{\sf E}}^{m+1}+\nu_{\ssy{\sf E}}^m\right)
+2\,C_2\,\tau\,(\tau^2+h^2),\quad m=1,\dots,N-1.
\end{split}
\label{L2RECURSIVE}
\end{gather}
%
%
%
\par\noindent
\vskip0.2truecm
\par
{\sc Step 3}. Apply the operator $\partial_h$ on \eqref{Error_Equations_a}
and \eqref{Error_Equations}, and then use \eqref{PROP_commut}, \eqref{PROP_commut2}
and \eqref{PROD_DIFF} to get
\begin{gather}
\tfrac{{\sf D}^1-{\sf D}^0}{\tau}=
{\rm i}\,\rho\,\Delta_h\left(\tfrac{{\sf D}^1+{\sf D}^0}{2}\right)
-\sigma\,\partial_{h}\Delta_h\left(\tfrac{{\sf D}^1+{\sf D}^0}{2}\right)
+\sum_{\kappa=1}^9{\sf B}_{\kappa,0},\label{Error_Equations_Ba}\\
\tfrac{{\sf D}^{n+1}-{\sf D}^{n-1}}{2\,\tau}=
{\rm i}\,\rho\,\Delta_{h}\left(\tfrac{{\sf D}^{n+1}+{\sf D}^{n-1}}{2}\right)
-\sigma\,\partial_h\Delta_h\left(\tfrac{{\sf D}^{n+1}+{\sf D}^{n-1}}{2}\right)
+\sum_{\kappa=1}^9{\sf B}_{\kappa,n},\quad n=1,\dots,N-1,\label{Error_Equations_B}
\end{gather}
where
\begin{equation*}
\begin{split}
{\sf B}_{1,n}:=&\,-\tfrac{3\alpha}{2}\,{\mathcal
A}_h\left[\,\sigma^{-}\left(\,|\phi^{n}|^2-|{\mathfrak m}_n|^2
\,\right)\otimes\partial_h^2\left(\phi^{\ell(n)}+\phi^{r(n)}\right)\,\right],\\
{\sf B}_{2,n}:=&\,-\tfrac{3\alpha}{2}\,{\mathcal
A}_h\left[\,\partial_h\left(\,|\phi^{n}|^2-|{\mathfrak m}_n|^2\,\right)
\otimes\sigma^{+}\left(\phi^{\ell(n)}+\phi^{r(n)}\right)\,\right],\\
{\sf B}_{3,n}:=&\,-\tfrac{3\alpha}{2}\,{\mathcal A}_h
\left[\,\sigma^{-}(|{\mathfrak m}_n|^2)\otimes
\partial_h\left({\sf D}^{\ell(n)}+{\sf D}^{r(n)}\right)\,\right],\\
{\sf B}_{4,n}:=&\,-\tfrac{3\alpha}{2}\,{\mathcal A}_h\left[\,
\partial_h(|{\mathfrak m}_n|^2)
\otimes\sigma^{+}\left({\sf D}^{\ell(n)}+{\sf D}^{r(n)}\right)\,\right],\\
{\sf B}_{5,n}:=&\,{\rm i}\,\tfrac{\delta}{2}\,\left[\,
\sigma^{-}(|\phi^n|^2-|{\mathfrak m}_n|^2)
\otimes \partial_h\left(\phi^{\ell(n)}+\phi^{r(n)}\right)\,\right],\\
{\sf B}_{6,n}:=&\,{\rm i}\,\tfrac{\delta}{2}\,\left[\,
\partial_h(|\phi^n|^2-|{\mathfrak m}_n|^2)
\otimes \sigma^{+}\left(\phi^{\ell(n)}+\phi^{r(n)}\right)\,\right],\\
{\sf B}_{7,n}:=&\,{\rm i}\,\tfrac{\delta}{2}\,\left[\,
\sigma^{-}(|{\mathfrak m}_n|^2)\otimes\left({\sf D}^{\ell(n)}
+{\sf D}^{r(n)}\right)\,\right],\\
{\sf B}_{8,n}:=&\,{\rm i}\,\tfrac{\delta}{2}\,\left[\,
\partial_h(|{\mathfrak m}_n|^2)\otimes\sigma^{+}\left({\sf E}^{\ell(n)}
+{\sf E}^{r(n)}\right)\,\right],\\
{\sf B}_{9,n}:=&\,\partial_h\eta^n.
\end{split}
\end{equation*}
%
%
%
%
\par\noindent
\medskip
\par
{\sc Step 4}. First, we take the inner product
$(\cdot,\cdot)_{0,h}$
of \eqref{Error_Equations_Ba} with $({\sf D}^1+{\sf D}^0)$
and of \eqref{Error_Equations_B} with $({\sf D}^{n+1}+{\sf D}^{n-1})$.
Then, we take real parts and use the
properties \eqref{PROP_II} and \eqref{PROP_III}, to get
\begin{gather}
\|{\sf D}^1\|_{0,h}^2-\|{\sf D}^0\|_{0,h}^2=
\tau\,\sum_{\kappa=1}^9b_{\kappa,0},\label{xioni1_a}\\
\|{\sf D}^{n+1}\|_{0,h}^2-\|{\sf D}^{n-1}\|_{0,h}^2=
2\tau\,\sum_{\kappa=1}^9b_{\kappa,n},\quad n=1,\dots,N-1,\label{xioni1}
\end{gather}
where
\begin{equation*}
b_{\kappa,n}:={\mathrm{Re}}\big({\sf B}_{\kappa,n},{\sf D}^{\ell(n)}+{\sf D}^{r(n)}\big)_{0,h}.
\end{equation*}
\par
Let $n\in\{0,\dots,N-1\}$. It is obvious that
\begin{gather}
b_{7,n}=0,\label{xioni2}\\
b_{9,n}\leq\,|\eta^n|_{1,h}\,\left[\,\|{\sf D}^{\ell(n)}\|_{0,h}
+\|{\sf D}^{r(n)}\|_{0,h}\,\right].\label{xioni2a}
\end{gather}
Now, using the Cauchy-Schwarz inequality, \eqref{PROP_V}, \eqref{BPROP_I}, \eqref{NewEra3},
\eqref{mol2}, \eqref{NewEra4}, \eqref{BPROP_III}
and \eqref{G_Lip} (with $\nu=\|\cdot\|_{0,h}$), we have
\begin{equation}\label{xioni3}
\begin{split}
b_{4,n}+b_{8,n}\leq&\,\tfrac{|\delta|}{2}\,|\partial_h(|{\mathfrak m}_n|^2)|_{\infty,h}
\,\|\sigma^{+}({\sf E}^{\ell(n)}+{\sf E}^{r(n)})\|_{0,h}
\,\|{\sf D}^{\ell(n)}+{\sf D}^{r(n)}\|_{0,h}\\
&\quad+\tfrac{3|\alpha|}{2}\,|\partial_h(|{\mathfrak m}_n|^2)|_{\infty,h}
\,\|\sigma^{+}({\sf D}^{\ell(n)}+{\sf D}^{r(n)})\|_{0,h}
\,\|{\sf D}^{\ell(n)}+{\sf D}^{r(n)}\|_{0,h}\\
\leq&\,C\,\|{\mathfrak m}_n\|_{1,\infty,h}^2
\,\|{\sf E}^{\ell(n)}+{\sf E}^{r(n)}\|_{0,h}
\,\|{\sf D}^{\ell(n)}+{\sf D}^{r(n)}\|_{0,h}\\
&\quad+C\,\|{\mathfrak m}_n\|_{1,\infty,h}^2
\,\|{\sf D}^{\ell(n)}+{\sf D}^{r(n)}\|_{0,h}^2\\
\leq&\,C\,\lambda_{\sf c}^2\,\left[\,\|{\sf E}^{\ell(n)}+{\sf E}^{r(n)}\|_{0,h}
+\|{\sf D}^{\ell(n)}+{\sf D}^{r(n)}\|_{0,h}\,\right]
\,\left[\,\|{\sf D}^{\ell(n)}\|_{0,h}+\|{\sf D}^{r(n)}\|_{0,h}\,\right]\\
\end{split}
\end{equation}
and
\begin{equation}\label{xioni5}
\begin{split}
b_{1,n}+b_{5,n}\leq&\,\tfrac{3|\alpha|}{2}
\,\big|\partial_h^2\big(\phi^{\ell(n)}+\phi^{r(n)}\big)\big|_{\infty,h}
\,\|\sigma^{-}(|\phi^n|^2-|{\mathfrak m}_n|^2)\|_{0,h}
\,\|{\sf D}^{\ell(n)}+{\sf D}^{r(n)}\|_{0,h}\\
&\quad+\tfrac{|\delta|}{2}
\,\big|\partial_h\big(\phi^{\ell(n)}+\phi^{r(n)}\big)\big|_{\infty,h}
\,\|\sigma^{-}(|\phi^n|^2-|{\mathfrak m}_n|^2)\|_{0,h}
\,\|{\sf D}^{\ell(n)}+{\sf D}^{r(n)}\|_{0,h}\\
\leq&\,C\,\mu_{\star}\,\| |\phi^n|^2-|{\mathfrak m}_n|^2\|_{0,h}
\,\|{\sf D}^{\ell(n)}+{\sf D}^{r(n)}\|_{0,h}\\
\leq&\,C\,\mu_{\star}\,\left(|\phi^n|_{\infty,h}+|{\mathfrak m}_n|_{\infty,h}\right)
\,\|\phi^n-{\mathfrak m}_n\|_{0,h}
\,\|{\sf D}^{\ell(n)}+{\sf D}^{r(n)}\|_{0,h}\\
\leq&\,C\,\lambda_{\sf c}^2\,\|{\sf E}^n\|_{0,h}
\,\left[\,\|{\sf D}^{\ell(n)}\|_{0,h}+\|{\sf D}^{r(n)}\|_{0,h}\,\right].\\
\end{split}
\end{equation}
Also, using the Cauchy-Schwarz inequality,
\eqref{PROP_V}, \eqref{PROD_DIFF}, \eqref{mol2}  and \eqref{G_Lip}
(with $\nu(\cdot)=\|\cdot\|_{0,h}$ or $\nu(\cdot)=\|\cdot\|_{1,h}$),
we obtain
%
%
%
\begin{equation*}\label{xioni8}
\begin{split}
b_{2,n}+b_{6,n}\leq&\,\left(\tfrac{3\,|\alpha|}{2}+\tfrac{|\delta|}{2}\right)
\,\big|\sigma^{+}\big(\phi^{\ell(n)}+\phi^{r(n)}\big)\big|_{\infty,h}
\,\|\partial_h\big(|\phi^n|^2-|{\mathfrak m}_n|^2\big)\|_{0,h}
\,\|{\sf D}^{\ell(n)}+{\sf D}^{r(n)}\|_{0,h}\\
\leq&\,C\,\mu_{\star}\,
\,\|\partial_h\big(|\phi^n|^2-|{\mathfrak m}_n|^2\big)\|_{0,h}
\,\|{\sf D}^{\ell(n)}+{\sf D}^{r(n)}\|_{0,h}\\
\end{split}
\end{equation*}
and
\begin{equation*}\label{xioni8}
\begin{split}
\|\partial_h\big(|\phi^n|^2-|{\mathfrak m}_n|^2\big)\|_{0,h}
\leq&\,\|\partial_h\big[(\phi^n-{\mathfrak m}_n)
\otimes({\overline{\phi^n+{\mathfrak m}_n})}\big]\|_{0,h}\\
\leq&\,\|\,\partial_h(\phi^n-{\mathfrak m}_n)
\otimes{\overline{\sigma^{+}(\phi^n+{\mathfrak m}_n)}}\,\|_{0,h}
+\|\,\sigma^{-}(\phi^n-{\mathfrak m}_n)
\otimes{\overline{\partial_h(\phi^n+{\mathfrak m}_n)}}\,\|_{0,h}\\
\leq&\,\left(\|\phi^n\|_{1,\infty,h}+\|{\mathfrak m}_n\|_{1,\infty,h}\right)\,
\left[\,\|\partial_h(\phi^n-{\mathfrak m}_n)\|_{0,h}
+\|\sigma^{-}(\phi^n-{\mathfrak m}_n)\|_{0,h}\,\right]\\
\leq&\,C\,(\mu_{\star}+3\,\lambda_{\sf c})\,
\left(\,|\phi^n-{\mathfrak m}_n|_{1,h}
+\|\phi^n-{\mathfrak m}_n\|_{0,h}\,\right)\\
\leq&\,C\,\lambda_{\sf c}\,
\left(\,|{\sf E}^n|_{1,h}+\|{\sf E}^n\|_{0,h}\,\right)\\
\leq&\,C\,\lambda_{\sf c}\,
\left(\,\|{\sf D}^n\|_{0,h}+\|{\sf E}^n\|_{0,h}\,\right),\\
\end{split}
\end{equation*}
which, finally, yield
\begin{equation}\label{xioni40}
b_{2,n}+b_{6,n}\leq\,C\,\lambda_{\sf c}^2\,
\left(\,\|{\sf D}^n\|_{0,h}+\|{\sf E}^n\|_{0,h}\,\right)
\,\left[\,\|{\sf D}^{\ell(n)}\|_{0,h}+\|{\sf D}^{r(n)}\|_{0,h}\,\right].
\end{equation}
Combining \eqref{PROP_VII}, \eqref{BPROP_I} and \eqref{mol2}, we have
\begin{equation}\label{xioni9}
\begin{split}
b_{3,n}=&\,-\tfrac{3\alpha}{2}\,{\rm Re}\left({\mathcal A}_h
\big[\,(|\sigma^{-}({\mathfrak m}_n)|^2)\otimes
\partial_h\big({\sf D}^{\ell(n)}+{\sf D}^{r(n)}\big)\,\big],
{\sf D}^{\ell(n)}+{\sf D}^{r(n)}\right)_{0,h}\\
\leq&\,\tfrac{3}{4}\,|\alpha|
\,\big|\big(\partial_h(|\sigma^{-}({\mathfrak m}_n)|^2),
|{\sf D}^{\ell(n)}+{\sf D}^{r(n)}|^2\big)_{0,h}\big|\\
\leq&\,C\,\big|\partial_h(|\sigma^{-}({\mathfrak m}_n)|^2)\big|_{\infty,h}
\,\|{\sf D}^{\ell(n)}+{\sf D}^{r(n)}\|_{0,h}^2\\
\leq&\,C\,|\sigma^{-}(\partial_h(|{\mathfrak m}_n|^2))|_{\infty,h}
\,\|{\sf D}^{\ell(n)}+{\sf D}^{r(n)}\|_{0,h}^2\\
\leq&\,C\,|\partial_h(|{\mathfrak m}_n|^2)|_{\infty,h}
\,\|{\sf D}^{\ell(n)}+{\sf D}^{r(n)}\|_{0,h}^2\\
\leq&\,C\,\|{\mathfrak m}_n\|_{1,\infty,h}^2
\,\|{\sf D}^{\ell(n)}+{\sf D}^{r(n)}\|_{0,h}^2\\
\leq&\,C\,\lambda_{\sf c}^2
\,\|{\sf D}^{\ell(n)}+{\sf D}^{r(n)}\|_{0,h}
\,\left[\,\|{\sf D}^{\ell(n)}\|_{0,h}+\|{\sf D}^{r(n)}\|_{0,h}\,\right].\\
\end{split}
\end{equation}
\par
Letting ${\nu}_{\ssy{\sf D}}^m:=\|{\sf D}^m\|_{0,h}+\|{\sf D}^{m-1}\|_{0,h}$ for $m=1,\dots,N$,
observing that ${\sf E}^0={\sf D}^0=0$, and using \eqref{xioni1_a} and \eqref{xioni1}
along with \eqref{xioni2}, \eqref{xioni2a}, \eqref{xioni3}, \eqref{xioni5}, \eqref{xioni40},
\eqref{xioni9}, \eqref{consistency_bound0} and \eqref{consistency_bound},
we conclude that there exist positive constants $C_{\ssy{\sf D},1}$ and
$C_{\ssy{\sf D},2}$ such that
\begin{gather}
\nu_{\ssy{\sf D}}^1\leq\,C_{\ssy{\sf D},1}\,\lambda_{\sf c}^2\,\tau\,
\,\left(\nu_{\ssy{\sf E}}^1+\nu_{\ssy{\sf D}}^1\right)
+C_1\,(\tau^2+\tau\,h^2),\label{H1RECURSIVEa}\\
\begin{split}
\nu_{\ssy{\sf D}}^{m+1}\leq\nu_{\ssy{\sf D}}^m
&\,+C_{\ssy{\sf D},2}\,\lambda_{\sf c}^2\,\tau\,
\,\left(\nu_{\ssy{\sf E}}^{m+1}+\nu_{\ssy{\sf E}}^m
+\nu_{\ssy{\sf D}}^m+\nu_{\ssy{\sf D}}^{m+1}\right)\\
&\,+2\,C_2\,\tau\,(\tau^2+h^2),
\quad m=1,\dots,N-1.\\
\end{split}
\label{H1RECURSIVE}
\end{gather}
%
%
%
\par\noindent
\vskip0.2truecm
\par
{\sc Step 5}. Let $C_{\star}=\max\{{\sf C}_1,C_{\ssy{\sf E},1},
C_{\ssy{\sf E},2},C_{\ssy{\sf D},1},C_{\ssy{\sf D},2}\}$, and
$\nu^m:=\nu_{\ssy{\sf E}}^m+\nu_{\ssy{\sf D}}^m$ for
$m=1,\dots,N$. Assuming that $4\,\tau\,\lambda_{\sf c}^2\,C_{\star}<1$,
and using the inequalities \eqref{L2RECURSIVEa},
\eqref{L2RECURSIVE}, \eqref{H1RECURSIVEa} and \eqref{H1RECURSIVE},
we conclude that
\begin{gather}
\nu^1\leq\,4\,C_1\,(\tau^2+\tau\,h^2),\label{tzou1}\\
\nu^{m+1}\leq\,\tfrac{1+2\,C_{\star}\,\lambda_{\sf c}^2\,\tau}
{1-2\,C_{\star}\,\lambda_{\sf c}^2\,\tau}\,\nu^m
+8\,C_2\,\tau\,(\tau^2+h^2),\quad m=1,\dots,N-1.\label{tzou2}
\end{gather}
The estimate \eqref{mainre} follows easily by employing a standard
discrete Gronwall argument based on \eqref{tzou1} and \eqref{tzou2}.
\end{proof}
%
%
%
\section{Convergence of the (FDM) Approximations}\label{Section4}
%
Using the convergence result of Theorem~\ref{Conv_Of_Mod},
we are able to find a mild mesh condition, which when satisfied ensures that
the $\|\cdot\|_{1,\infty,h}-$distance of the (MFDS) approximations from the exact
solution to the continuous problem is bounded by $\lambda$, for a given value of $\lambda$.
%
%
\begin{proposition}\label{UUmod}
Let $\mu_{\star}:=\max_{0\leq{\ell}\leq{2}}
\left(\max_{\ssy{[0,1]\times[0,T]}}|\partial_x^{\ell}\phi|\right)$,
$\lambda_{{\sf c}}:=\mu_{\star}+1$,
${\sf C}_1$ be the constant specified in Proposition~\ref{ModExist}
and ${\sf C}_2,{\sf C}_3$ be the constants
specified in Proposition~\ref{Conv_Of_Mod}, where ${\sf C}_3\ge{\sf C}_1$.
Also, we assume that $\tau\,{\sf C}_3\,\lambda_{{\sf c}}^2<1$
and let $({\sf Z}^{\ell})_{\ell=0}^{\ssy N}$
be the modified finite difference approximations defined by
\eqref{modfidifin}--\eqref{modscheme1} for
$\lambda=\lambda_{{\sf c}}$, i.e.,
${\sf Z}^{\ell}=S^{\ell}(\lambda_{{\sf c}})$ for $\ell=0,\dots,N$. 
If
\begin{equation}\label{MeshCond1}
{\sf C}_2\,(\tau^2\,h^{-\frac{1}{2}}+h^{\frac{3}{2}})
\leq\,\lambda_{\sf c},
\end{equation}
then, it holds that
\begin{equation}\label{Mod_U_Estim}
\max_{0\leq{m}\leq{\ssy N}}\|{\sf Z}^m-\phi^m\|_{1,\infty,h}\leq\,\lambda_{\sf c}
\end{equation}
and
\begin{equation}\label{EttStep}
{\mathfrak m}(\lambda_{\sf c},t_m;{\sf Z}^m)={\sf Z}^m,\quad m=0,\dots,N.
\end{equation}
\end{proposition}
%
%
%
%
\begin{proof}
The convergence estimate \eqref{mainre}, the inverse inequality
\eqref{maxL2} and \eqref{MeshCond1} yield that
\begin{equation*}
\begin{split}
\|{\sf Z}^m-\phi^m\|_{1,\infty,h}=&\,\max\left\{|{\sf Z}^m-\phi^m|_{\infty,h},
|\partial_h({\sf Z}^m-\phi^m)|_{\infty,h}\right\}\\
\leq&\,h^{-\frac{1}{2}}\,\max\left\{\|{\sf Z}^m-\phi^m\|_{0,h},
\|\partial_h({\sf Z}^m-\phi^m)\|_{0,h}\right\}\\
\leq&\,h^{-\frac{1}{2}}\,\|{\sf Z}^m-\phi^m\|_{1,h}\\
\leq&\,{\sf C}_2\,(\tau^2\,h^{-\frac{1}{2}}+h^{\frac{3}{2}})\\
\leq&\,\lambda_{\sf c},
\end{split}
\end{equation*}
which establish \eqref{Mod_U_Estim}. Finally, \eqref{EttStep}
follows as a simple consequence of \eqref{Mod_U_Estim} and of
definitions \eqref{xi_function} and \eqref{G_function}.
\end{proof}
%
%
%
%
%
Now, we are ready to show that the  (FDM) approximations are well-defined and have
second order convergence with respect to  $\|\cdot\|_{\infty,h}-$norm.
%
%
%
\begin{thm}\label{ConvergenceRate}
Let $\mu_{\star}:=\max_{0\leq{\ell}\leq{2}}
\left(\max_{\ssy{[0,1]\times[0,T]}}|\partial_x^{\ell}\phi|\right)$,
$\lambda_{{\sf c}}:=\mu_{\star}+1$,
${\sf C}_1$ be the constant specified in Proposition~\ref{ModExist}
and ${\sf C}_2, {\sf C}_3$ be the constants
specified in Proposition~\ref{Conv_Of_Mod} where ${\sf C}_3\ge{\sf C}_1$.
Also, we assume that $\tau\,{\sf C}_3\,\lambda_{{\sf c}}^2<1$
and that \eqref{MeshCond1} holds. Then, the
finite difference method
\eqref{fidifin}-\eqref{scheme1} is well-defined
and
\begin{equation}\label{F_Estimate}
\max_{0\leq{m}\leq{\ssy
N}}\|\phi^m-\Phi^m\|_{1,h}\leq\,{\sf C}_2\,(\tau^2+h^2).
\end{equation}
\end{thm}
%
%
%
%
%
\begin{proof}
Since we have $\lambda_{\star}\leq\mu_{\star}<\lambda_{\sf c}$
and $\tau\,{\sf C}_1\,\lambda_{\sf c}^2\leq\tau\,{\sf C}_3\,\lambda_{\sf c}^2<1$,
Proposition~\ref{ModExist} yields that
 the (MFDS) approximations \eqref{modfidifin}--\eqref{modscheme1}
are well-defined when $\lambda=\lambda_{\sf c}$. Then, we simplify
the notation by setting ${\sf Z}^{\ell}:=S^{\ell}(\lambda_{{\sf c}})$ for
$\ell=0,\dots,N$.
Also, by assuming that $\tau\,{\sf C}_3\,\lambda_{\sf c}^2<1$ and \eqref{MeshCond1}
hold, Proposition~\ref{UUmod} gives that \eqref{EttStep} and \eqref{Mod_U_Estim} are
satisfied.
\par
Combining \eqref{EttStep} with \eqref{modscheme1a} and \eqref{modscheme1},
we conclude that the (MFDS) approximations $({\sf Z}^{\ell})_{\ell=0}^{\ssy N}$
are also (FDM) approximations, i.e.,
\eqref{fidifin}, \eqref{scheme11} and \eqref{scheme1} hold, after setting
$\Phi^{\ell}={\sf Z}^{\ell}$ for $\ell=0,\dots,N$.
\par
Let $(\Psi^m)_{m=0}^{\ssy N}$ be an outcome of the finite difference method
\eqref{fidifin}-\eqref{scheme1}, and set ${\sf e}^m:={\sf Z}^m-\Psi^m$ for $m=0,\dots,N$.
We will show, by induction, that ${\sf e}^m=0$ for $m=0,\dots,N$. Since the initial
value is common, we have ${\sf e}^0=0$.
Thus, using \eqref{scheme11}, we obtain
\begin{equation}\label{canoo}
2\,{\sf e}^1={\rm i}\,\rho\,\tau\,\Delta_h{\sf e}^1
-\sigma\,\tau\,\partial_h\Delta_h{\sf e}^1
-3\,\alpha\,\tau\,{\mathcal A}_h\left(|\phi^0|^2\otimes\partial_h{\sf e}^1\right)
+{\rm i}\,\delta\,\tau\,\left(|\phi^0|^2\otimes{\sf e}^1\right).
\end{equation}
First, take the $(\cdot,\cdot)_{0,h}-$inner product of \eqref{canoo}
with ${\sf e}^1$, and then, keep real parts and use \eqref{PROP_II}, \eqref{PROP_III},
\eqref{PROP_VII} and \eqref{BPROP_I}, to obtain
\begin{equation*}
\begin{split}
2\,\|{\sf e}^1\|_{0,h}^2=&\,\tfrac{3\,\alpha}{2}\,\tau\,\left(\partial_h\left(|\phi^0|^2\right),
|{\sf e}^1|^2\right)_{0,h}\\
\leq&\,3\,|\alpha|\,\tau\,|\phi^0|_{\infty,h}\,|\phi^0|_{1,\infty,h}\,\|{\sf e}^1\|_{0,h}^2\\
\leq&\,3\,|\alpha|\,\tau\,\|\phi^0\|_{1,\infty,h}^2\,\|{\sf e}^1\|_{0,h}^2\\
\leq&\,3\,|\alpha|\,\lambda_{\sf c}^2\,\tau\,\|{\sf e}^1\|_{0,h}^2,\\
\end{split}
\end{equation*} 
which yields
\begin{equation}\label{canoo3}
\|{\sf e}^1\|_{0,h}^2\,\left(2-3\,|\alpha|\,\lambda_{\sf c}^2\,\tau\right)\leq 0.
\end{equation}
Recalling from the proof of Proposition~\ref{ModExist} that ${\sf C}_1=27\,|\alpha|$, we have
\begin{equation*}
3\,|\alpha|\,\lambda_{\sf c}^2\,\tau\leq\tau\,{\sf C}_1\,\lambda_{\sf c}^2
\leq\,\tau\,{\sf C}_3\,\lambda_{\sf c}^2<1.
\end{equation*}
Thus, from \eqref{canoo3}, follows that ${\sf e}^1=0$.
Let $\kappa\in\{1,\dots,N-1\}$ and that ${\sf e}^{\ell}=0$
for $\ell=0,\dots,\kappa$. Then, using \eqref{scheme1}, we have 
\begin{equation}\label{canoo1}
{\sf e}^{\kappa+1}={\rm i}\,\rho\,\tau\,\Delta_h{\sf e}^{\kappa+1}
-\sigma\,\tau\,\partial_h\Delta_h{\sf e}^{\kappa+1}
-3\,\alpha\,\tau\,{\mathcal A}_h\left(|{\sf Z}^{\kappa}|^2
\otimes\partial_h{\sf e}^{\kappa+1}\right)
+{\rm i}\,\delta\,\tau\,\left(|{\sf Z}^{\kappa}|^2\otimes{\sf e}^{\kappa+1}\right).
\end{equation}
Taking the $(\cdot,\cdot)_{0,h}-$inner product of \eqref{canoo1} by ${\sf e}^{\kappa+1}$
and then keeping real parts and using \eqref{PROP_II}, \eqref{PROP_III},
\eqref{PROP_VII}, \eqref{BPROP_I} and \eqref{Mod_U_Estim}, we get
\begin{equation*}
\begin{split}
\|{\sf e}^{\kappa+1}\|_{0,h}^2
=&\,\tfrac{3\,\alpha}{2}\,\tau\,\left(\partial_h\left(|{\sf Z}^{\kappa}|^2\right),|{\sf e}^{\kappa+1}|^2\right)_{0,h}\\
\leq&\,3\,|\alpha|\,\tau\,\|{\sf Z}^{\kappa}\|_{1,\infty,h}^2
\,\|{\sf e}^{\kappa+1}\|_{0,h}^2\\
\leq&\,3\,|\alpha|\,\left(\|{\sf Z}^{\kappa}-\phi^{\kappa}\|_{1,\infty,h}
+\|\phi^{\kappa}\|_{1,\infty,h}\right)^2
\,\|{\sf e}^{\kappa+1}\|_{0,h}^2\\
\leq&\,12\,|\alpha|\,\lambda_{\sf c}^2
\,\|{\sf e}^{\kappa+1}\|_{0,h}^2,\\
\end{split}
\end{equation*} 
which yields
\begin{equation}\label{canoo4} 
\|{\sf e}^{\kappa+1}\|_{0,h}^2\,\left(1-12\,|\alpha|\,\lambda_{\sf c}^2\,\tau\right)\leq 0.
\end{equation}
Observing that
\begin{equation*}
12\,|\alpha|\,\lambda_{\sf c}^2\,\tau\leq\tau\,{\sf C}_1\,\lambda_{\sf c}^2
\leq\,\tau\,{\sf C}_3\,\lambda_{\sf c}^2<1,
\end{equation*}
\eqref{canoo4} yields that ${\sf e}^{\kappa+1}=0$, which closes the induction argument.
\par
Under our assumptions on $\tau$ and $h$, we have shown
that the (FDM) approximations are well-defined and, in particular,
that $\Phi^m={\sf Z}^m$ for $m=0,\dots,N$.
Thus, the error estimate \eqref{F_Estimate} follows from
the error bound \eqref{mainre} for the (MFDS) approximations.
\end{proof}
%
%
%
\section{Numerical Results}\label{Section5}
We implemented the (FDM) in a {\tt FORTRAN 90}
program named {\tt  FD}. The program {\tt  FD} uses double precision complex arithmetic and
solves, at every time-step, the resulting cyclic penta-diagonal linear system of algebraic equations
using  a direct method based on the well-known Gauss eli\-mi\-nation for banded matrices.
For graph drawing, we used the {\tt gnuplot} command-line program \cite{gnuplot}.
When the exact solution is known, we measured the approximation error in the
space-time discrete maximum norm:
\begin{equation*}
E_{\infty}(N,J):=\max_{0\leq{n}\leq{\ssy N}}\,\max_{1\leq{j}\leq{\ssy J}}\big|\Phi^n_j-\phi_j^n\big|.
\end{equation*}
Also, letting $N$ be proportional to $J$ (i.e. $N=q\,J$ for a given $q\in{\mathbb Q}$), we computed
the experimental order of convergence for successive values values $J_1$ and $J_2$ of $J$,
using the formula
\begin{equation*}
\log\left(E_{\infty}(q\,J_1,J_1)/E_{\infty}(q\,J_2,J_2)\right)/\log(J_2/J_1).
\end{equation*}
%
%
\subsection{Example 1}
We consider the problem \eqref{HirotaPeriodic}-\eqref{HirotaPeriodic_b} with:
${\sf L}=1$, $(\rho,\alpha,\sigma,\delta)=\left(2,\frac{1}{2},4,\frac{1}{4}\right)$
or $(\rho,\alpha,\sigma,\delta)=\left(0,\frac{1}{2},4,0\right)$,
and load $f$ such that the function $\phi(t,x)=e^{{\rm i}\,(t+2\,\pi\,x)}$
to be its exact solution.  In the perfomed numerical experiments we have
set $T=5$, $[x_a,x_b]=[0,1]$, $(N,J)=(5\nu,\nu)$ for  $\nu=20,40,80,160,320,640$,
and computed the approximation error $E_{\infty}(N,J)$. The results shown on Table~\ref{Table1}
confirm that the experimental order of convergence with respect to $\frac{1}{\nu}$
is equal to $2$, which is in agreement with Theorem~\ref{ConvergenceRate}.
\begin{table}[htp]
\begin{tabular}{|c|c|c|c|}
\hline
\multicolumn{3}{|c|}{$(\rho,\alpha,\sigma,\delta)=(2,0.5,4,0.25)$}\\
\hline
$\nu$    &   $E_{\infty}(5\nu,\nu)$    &   Rate \\
\hline
%
$20$   & 6.54162(-2)                        & ---  \\
\hline
$40$  & 1.61352(-2)                         & 2.019\\
\hline
$80$ & 4.01170(-3)                          & 2.008\\
\hline
$160$ & 9.91777(-4)                        & 2.016\\
\hline
$320$ & 2.39596(-4)                        & 2.049\\
\hline
$640$ & 5.63974(-5)                        & 2.087\\
\hline
\end{tabular}
\quad
\begin{tabular}{|c|c|c|c|}
\hline
\multicolumn{3}{|c|}{$(\rho,\alpha,\sigma,\delta)=(0,0.5,4,0)$}\\
\hline
$\nu$    &   $E_{\infty}(5\nu,\nu)$    &   Rate \\
\hline
%
$20$   & 6.20065(-2)                        & ---  \\
\hline
$40$  & 1.52387(-2)                         & 2.025\\
\hline
$80$ & 3.80623(-3)                          & 2.001\\
\hline
$160$ & 9.42224(-4)                        & 2.014\\
\hline
$320$ & 2.28071(-4)                        & 2.047\\
\hline
$640$ & 5.37631(-5)                        & 2.085\\
\hline
\end{tabular}
\par\noindent\vskip0.2truecm\par\noindent
\caption{Discrete maximum norm convergence for Example 1.}
\label{Table1}
\end{table}
\subsection{Example 2}\label{specialsolutions}
For $\kappa\in{\mathbb R}$, it is easily seen that
the function $\phi:[0,T]\times{\mathbb R}\rightarrow{\mathbb C}$ given by
\begin{equation}\label{Exactsol}
\phi(t,x)=e^{{\rm i}\,(\kappa\,x+\omega\,t)}\quad\forall\,
(t,x)\in[0,T]\times{\mathbb R}
\end{equation}
with
\begin{equation}\label{ExactSol_b}
\omega=\kappa^2\,(\sigma\,\kappa-\rho)+\delta-3\,\alpha\,\kappa,
\end{equation}
solves of the homogeneous (SH) equation \eqref{HirotaPeriodic}
and is ${\sf L}-$periodic on ${\mathbb R}$ when
\begin{equation}\label{ExactSol_c}
\tfrac{\kappa\,{\sf L}}{2\,\pi}\in{\mathbb Z}.
\end{equation}
According to  \cite{Karan},  the smooth function \eqref{Exactsol} under \eqref{ExactSol_b} and
\eqref{ExactSol_c} is the unique ${\sf L}-$periodic solution to the problem \eqref{HirotaPeriodic}-\eqref{HirotaPeriodic_b}
with initial condition $\phi_0(x)=e^{{\rm i}\,\kappa\,x}$. Observe that the solution does not requires 
the condition $\rho\alpha=\sigma\delta$ to be valid.
\par
We performed numerical experiments choosing  ${\sf L}=1$, $\kappa=2\,\pi$, $T=1$, 
$[x_a,x_b]=[0,1]$, $(\rho,\alpha,\sigma,\delta)=\left(0,\tfrac{1}{2},\tfrac{1}{8},0\right)$
or $(\rho,\alpha,\sigma,\delta)=\left(\tfrac{1}{8},\tfrac{1}{2},\tfrac{1}{8},1\right)$, $(N,J)=(5\nu,\nu)$
for  $\nu=80,160,320,640,1280$, and computing the approximation error $E_{\infty}(N,J)$.
The results exposed on Table~\ref{Table2} show that the experimental order of convergence
with respect to $\frac{1}{\nu}$ is equal to $2$, which is in agreement with the convergence
result of Theorem~\ref{ConvergenceRate}. 
Also, we observe that the numerical method is efficient when
$\rho\alpha\not=\sigma\delta$, which is due to
the existence of a unique smooth solution.
\begin{table}[htp]
\begin{tabular}{|c|c|c|c|}
\hline
\multicolumn{3}{|c|}{$(\rho,\alpha,\sigma,\delta)=(0,0.5,0.125,0)$}\\
\hline
$\nu$    &   $E_{\infty}(5\nu,\nu)$    &   Rate \\
\hline
%
$80$   & 2.99534(-2)                        & ---  \\
\hline
$160$  & 7.49695(-3)                         & 1.998\\
\hline
$320$ & 1.87476(-3)                          & 1.999\\
\hline
$640$ & 4.68724(-4)                        & 1.999\\
\hline
$1280$ & 1.17186(-4)                        & 1.999\\
\hline
\end{tabular}
\quad
\begin{tabular}{|c|c|c|c|}
\hline
\multicolumn{3}{|c|}{$(\rho,\alpha,\sigma,\delta)=(0.125,0.5,0.125,1)$}\\
\hline
$\nu$    &   $E_{\infty}(5\nu,\nu)$    &   Rate \\
\hline
%
$80$   & 1.79621(-2)                        & ---  \\
\hline
$160$  & 4.49193(-3)                         & 1.999\\
\hline
$320$ & 1.12307(-3)                          & 1.999\\
\hline
$640$ & 2.80772(-4)                        & 1.999\\
\hline
$1280$ & 7.01904(-5)                        & 2.000\\
\hline
\end{tabular}
\par\noindent\vskip0.2truecm\par\noindent
\caption{Discrete maximum norm convergence  for Example 2.}
\label{Table2}
\end{table}
%
\subsection{Bright Soliton solutions}  
%
A Bright Soliton (BS) solution to the (SH) equation \eqref{HirotaPeriodic}
is a function $\phi_{\star}:[0,+\infty)\times{\mathbb R}\rightarrow{\mathbb C}$ of the form
\begin{equation}\label{B_soliton}
\phi_{\star}(t,x)=A\,\,{\rm sech}(B\,(x-v\,t))\,\,e^{{\rm i}\,(\kappa\,x+\omega\,t)}
\quad\forall\,(t,x)\in[0,+\infty)\times{\mathbb R},
\end{equation}
where $A\not=0$, $B\not=0$, $\kappa$ and $\omega$ are real constants
(see, e.g. \cite{Karan}, \cite{Biswas}). It is easily seen that a (BS)
solutions exist, iff, 
\begin{equation*}
\rho\,\alpha=\sigma\,\delta,\quad
v=\sigma\,(B^2-3\,\kappa^2)+2\,\rho\,\kappa,\quad
\omega=\rho\,(B^2-\kappa^2)+\kappa\,\sigma\,(\kappa^2-3B^2),\quad
|B|=|A|\,\left(\tfrac{a}{2\sigma}\right)^{\frac{1}{2}},
\end{equation*}
where $A$, $\kappa$ are parameters.  Observing that
\begin{equation*}
|\phi_{\star}(t,x)|\leq\,2\,|A|\,e^{-|B|\,|x-v\,t|}\quad\forall\,(t,x)\in[0,+\infty)\times{\mathbb R},
\end{equation*}
and taking into account the properties of the hyperbolic function ${\tt sech}$,
we conclude that, for $t\in[0,+\infty)$, the modulus of a (BS) solution tends monotonically to $0$
when $x\rightarrow\pm\infty$. Thus $\phi_{\star}$ has no periodic structure and it can not be
an ${\sf L}-$periodic solution to the problem \eqref{HirotaPeriodic}-\eqref{HirotaPeriodic_b}.
However, using our finite difference method, we are able to simulate a (BS) solution on a set
$[0,T]\times[x_a,x_b]$ in which it has {\sl almost} compact support with respect to the space variable.
\par
We consider problem \eqref{HirotaPeriodic}-\eqref{HirotaPeriodic_b} with:
$T=5$, ${\sf L}=40$, $[x_a,x_b]=[-20,20]$, $(\rho,\alpha,\sigma,\delta)=(\tfrac{1}{2},2,1,1)$,
and initial condition $\phi_0(x)=\phi_{\star}(0,x)$, where the bright soliton parameters are
given by $A=\tfrac{1}{2}$ and $\kappa=\tfrac{1}{2}$. In the numerical experiment, we have
computed the  approximation error $E_{\infty}(N,J)$ with 
$(N,J)=(2\nu,\nu)$ for  $\nu=200,400,800,1600,3200$. The results presented on Table~\ref{Table3}
confirm, again an experimental order of convergence equal to $2$.  Finally, Figure~\ref{Graph1} shows
a good final time graph agreement of the (BS) solution along with its finite difference approximation obtained
for $(N,J)=(400,300)$.
\begin{table}[htp]
\begin{tabular}{|c|c|c|c|}
\hline
\multicolumn{3}{|c|}{$(\rho,\alpha,\sigma,\delta)=(0.5,2,1,1)$, $(A,\kappa)=(0.5,0.5)$}\\
\hline
$\nu$    &   $E_{\infty}(2\nu,\nu)$    &   Rate \\
\hline
%
$200$   & 2.75915(-2)                        & ---  \\
\hline
$400$  & 7.14955(-3)                         & 1.948\\
\hline
$800$ & 1.80531(-3)                          & 1.985\\
\hline
$1600$ & 4.53750(-4)                        & 1.992\\
\hline
$3200$ & 1.15731(-4)                        & 1.971\\
\hline
\end{tabular}
\par\noindent\vskip0.2truecm\par\noindent
\caption{ Errors and discrete maximum norm convergence rates  to a (BS) solution.}
\label{Table3}
\end{table}
\par\noindent
\begin{figure}[ht]
\centering
\includegraphics[height=4.0truecm, width=6.5truecm]{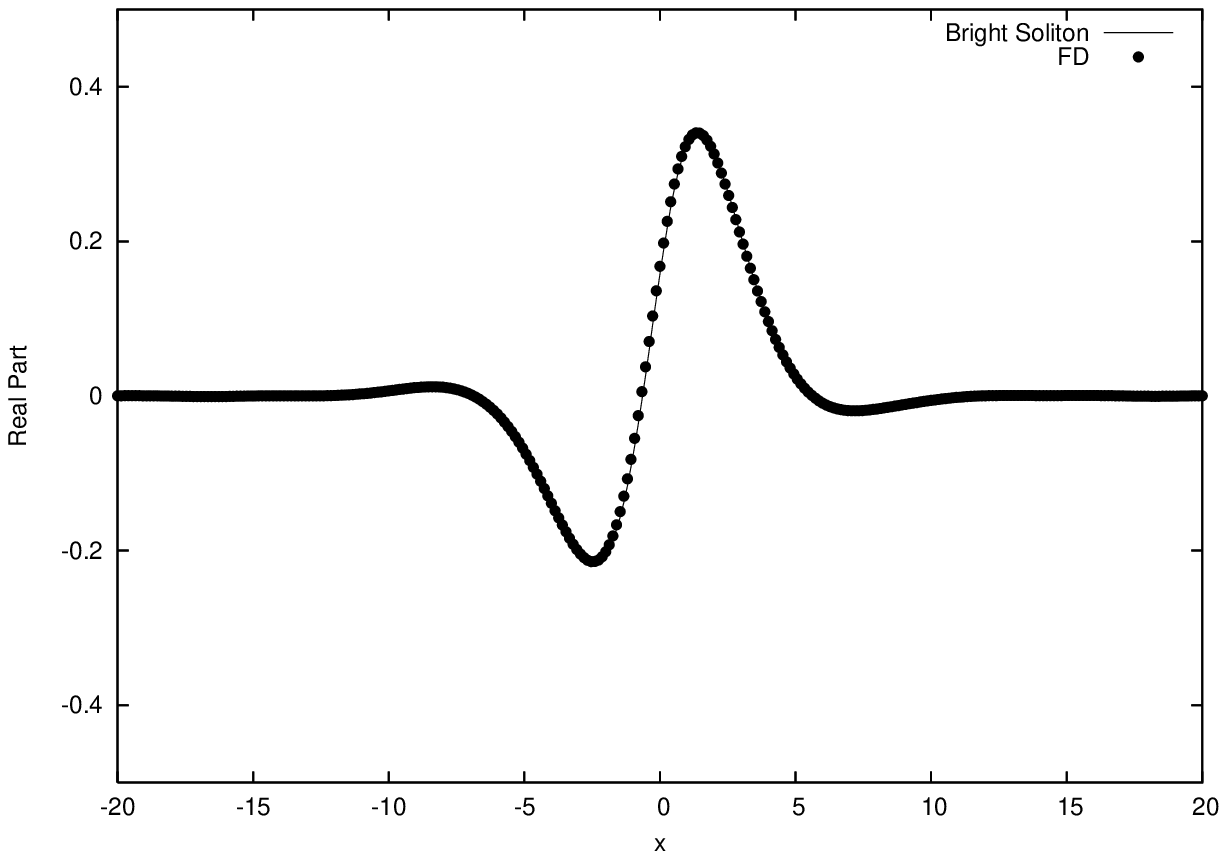}
\includegraphics[height=4.0truecm, width=6.5truecm]{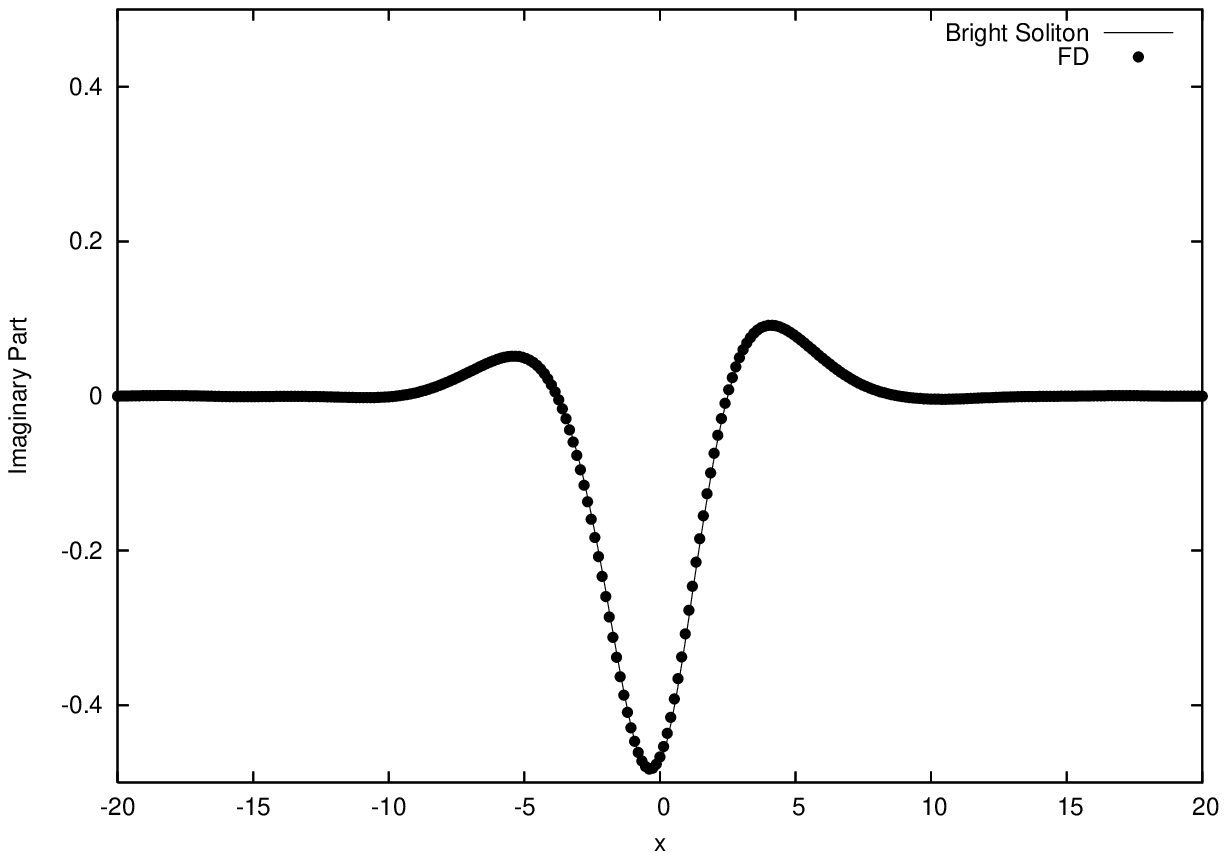}
\caption{Approximating a (BS) solution at $T=5$  with
$N=400$ and $J=300$.} \label{Graph1}
\end{figure}
%
%
%
%

\end{document}